\documentclass[11pt, a4paper]{amsart}

\usepackage[ansinew]{inputenc}
\usepackage[all]{xy}
\SelectTips{cm}{}
\usepackage[pagebackref,%colorlinks,linkcolor=black,citecolor=black%
  ,pdfborder={0 0 0}
  ,urlcolor=black,a4paper,hypertexnames=false]{hyperref}
\hypersetup{pdfauthor={Francesco Fournier-Facio, Clara L\"oh, Marco Moraschini}
  ,pdftitle=Bounded cohomology of finitely presented groups}

\usepackage{amsfonts}
\usepackage{amssymb, amsmath,amsthm, mathtools}
\usepackage{tabularx, hyperref}
\usepackage{enumerate}

\usepackage{tikz-cd}

\usepackage{stmaryrd}

\makeatletter
\newtheorem*{rep@theorem}{\rep@title}
\newcommand{\newreptheorem}[2]{%
\newenvironment{rep#1}[1]{%
 \def\rep@title{#2 \ref{##1}}%
 \begin{rep@theorem}}%
 {\end{rep@theorem}}}
\makeatother

\newtheorem{intro_thm}{Theorem}
\newtheorem{intro_cor}[intro_thm]{Corollary}

\newtheorem{intro_quest}[intro_thm]{Question}
\newtheorem{intro_defi}[intro_thm]{Definition}

\newtheorem{lemma}{Lemma}[section]
\newtheorem{thm}[lemma]{Theorem}
\newreptheorem{thm}{Theorem}  %for repeating the theorem
\newtheorem{prop}[lemma]{Proposition}
\newreptheorem{prop}{Proposition}  %for repeating the proposition
\newtheorem{cor}[lemma]{Corollary}
\newreptheorem{cor}{Corollary}  %for repeating the corollary

\theoremstyle{definition}
\newtheorem{defi}[lemma]{Definition}

\newreptheorem{quest}{Question}  %for repeating the question

\newtheorem{example}[lemma]{Example}
\newtheorem{rem}[lemma]{Remark}

\newtheorem{construction}[lemma]{Construction}
\newtheorem{scholium}[lemma]{Scholium}
\theoremstyle{definition}

\makeatletter
\newcommand\norm{\bBigg@{0.8}}

\makeatother
 
 \newcommand{\indnorm}[2][flex]{\csname #1l\endcsname\|#2%
                                 \csname #1r\endcsname\|\mathclose{}}
                                  \newcommand{\indnorml}[4][flex]{\csname #1l\endcsname\|#2%
                                 \csname #1r\endcsname\|_{#3}^{#4}\mathclose{}}

\newcommand{\genrel}[3][flex]{\csname #1l\endcsname\langle #2 \mathbin{\csname #1m\endcsname|} #3\csname #1r\endcsname\rangle}

\def\ltb#1{b^{(2)}_{#1}}
\def\loneb#1{\overline b^{\raisebox{-.2em}{$\ell^1$}}_{#1}}

\DeclareMathOperator{\BS}{BS}

\DeclareMathOperator{\lex}{\textup{Lex}}

\DeclareMathOperator{\cd}{cd}
\def\cdb{\cd_b}
\DeclareMathOperator{\bcd}{bcd}
\DeclareMathOperator{\hd}{hd}
\def\hdb{\hd_b}
\DeclareMathOperator{\projdim}{projdim}
\def\relprojdim{\projdim_b}
\DeclareMathOperator{\cost}{cost}
\DeclareMathOperator{\Ban}{Ban}
\DeclareMathOperator{\res}{res}

\DeclareMathOperator{\Groups}{{\bf Groups}}

\DeclareMathOperator{\id}{id}
\DeclareMathOperator{\Homeo}{Homeo}

\newcommand{\HH}{\operatorname{H}}

\newcommand{\CC}{\operatorname{C}}

\def\lonehom#1{\HH^{\ell^1}_{#1}}
\def\rlonehom#1{\overline{\HH}^{\ell^1}_{#1}}
\def\lonech#1{\CC^{\ell^1}_{#1}}

\newcommand{\bac}{\operatorname{BAc}}

\newcommand{\fa}[1]{%
  \forall_{#1}\quad}

\newcommand{\N}{\ensuremath {\mathbb{N}}}
\newcommand{\R} {\ensuremath {\mathbb{R}}}

\newcommand{\Z} {\ensuremath {\mathbb{Z}}}

\renewcommand{\rho}{\varrho}
\def\phi{\varphi}

\def\args{\;\cdot\;}

\usepackage{color}
\usepackage{pdfcolmk}

\long\def\forget#1{}

\def\longrightarrow{\rightarrow}
\def\longmapsto{\mapsto}

\makeatletter
\@namedef{subjclassname@2020}{%
  \textup{2020} Mathematics Subject Classification}
\makeatother

\begin{document}

\title[Bounded cohomology of finitely presented groups]{Bounded cohomology\\ of finitely presented groups:\\
    vanishing, non-vanishing, and computability}

\author[]{Francesco Fournier-Facio}
\address{Department of Mathematics, ETH Z\"urich, Z\"urich, Switzerland}
\email{francesco.fournier@math.ethz.ch}

\author[]{Clara L\"oh}
\address{Fakult\"{a}t f\"{u}r Mathematik, Universit\"{a}t Regensburg, Regensburg, Germany}
\email{clara.loeh@ur.de}

\author[]{Marco Moraschini}
\address{Fakult\"{a}t f\"{u}r Mathematik, Universit\"{a}t Regensburg, Regensburg, Germany \\
Current affiliation: Dipartimento di Matematica, Universit\`{a} di Bologna, Bologna, Italy}
\email{marco.moraschini2@unibo.it}

\thanks{}

\keywords{bounded cohomology, boundedly acyclic groups, mitotic groups, undecidability}
\subjclass[2020]{Primary: 18G90. Secondary: 20F10}
\date{\today.\ 
  Clara L\"oh and Marco Moraschini were supported by the CRC~1085 \emph{Higher Invariants}
  (Universit\"at Regensburg, funded by the~DFG).
  The results in this paper are part of Francesco Fournier-Facio's PhD project}

\begin{abstract}
  We provide new computations in bounded cohomology:

  A group is boundedly acyclic if its bounded cohomology with trivial
  real coefficients is zero in all positive degrees. We show that
  there exists a continuum of finitely generated non-amenable
  boundedly acyclic groups and
  construct a finitely presented non-amenable boundedly acyclic group.

  On the other hand, we construct a continuum of finitely generated
  groups, whose bounded cohomology has uncountable dimension in all
  degrees greater than or equal to~$2$, and a concrete finitely presented one.

  Countable non-amenable groups with these two extreme properties were
  previously known to exist, but these constitute the first finitely
  generated/finitely presented examples.

  Finally, we show that various algorithmic problems on bounded
  cohomology are undecidable.
\end{abstract}

\maketitle

%%%%%%%%%%%%%%%%%%%%%%%%%%%%%%%%%%%%%%%%%%%%%%%%%%%%%%%%%%%
\section{Introduction}

Bounded cohomology of groups is defined via the topological dual of
the simplicial resolution. This rich theory has applications to the
geometry of manifolds~\cite{vbc}, dynamics~\cite{Ghys}, rigidity
theory~\cite{rigidity, monodshalom}, quasimorphisms~\cite{Brooks,
  Grigorchuk} and stable commutator length~\cite{calegari}.  However,
beyond the case of amenable groups, computing the bounded cohomology
of a group is a very hard task, which can typically only be done in
low degrees.  We provide new computations in bounded cohomology of
finitely generated and finitely presented groups in arbitrarily large
degrees.

%%%%%%%%%
\subsection{Finitely generated non-amenable boundedly acyclic groups}

\emph{Boundedly acyclic groups} are those groups whose bounded
cohomology with trivial real coefficients vanishes in all positive degrees:

\begin{intro_defi}[Boundedly acyclic groups]\label{def:boundedly:acyclic:groups}
  A group $\Gamma$ is \emph{boundedly acyclic} if $\HH^n_b(\Gamma; \R)
  \cong 0$ for all $n \geq 1$.  Here, $\R$ denotes the real
  coefficients endowed with the trivial $\Gamma$-action.

  We write~$\bac$ for the class of boundedly acyclic groups.
\end{intro_defi}

The main examples of boundedly acyclic groups are amenable groups, as
proved by Johnson~\cite{Johnson}. Matsumoto and Morita showed that the
class of boundedly acyclic groups also contains non-amenable groups,
by proving that the group of homeomorphisms of $\R^n$ with compact
support has this property~\cite{MM}.  Similar techniques show that
there exist non-amenable boundedly acyclic groups that are countable:
all mitotic groups are boundedly acyclic~\cite{Loeh}. However, also
these examples are not finitely generated.

Recently, the interest in finding finitely presented bounded acyclic
groups has increased significantly because of the following
applications to spaces: The version of Gromov's Vanishing Theorem for
boundedly acyclic covers by Ivanov~\cite{Ivanov_bac_covers} and the
extended version of Gromov's Mapping Theorem~\cite{BAc}.

Combining mitoses and suitable HNN-extensions, we show:

\begin{intro_thm}[Finitely generated non-amenable boundedly acyclic groups; Theorem~\ref{thm:fg}]\label{thm:fg:bdd:acylic}
  There exists a functor $\mu \colon \Groups \to \Groups$ associating
  to each group $\Gamma$ a boundedly acyclic group $\mu(\Gamma)$ into
  which $\Gamma$ embeds.  The group $\mu(\Gamma)$ has the following
  properties:
\begin{enumerate}
\item Torsion elements of $\mu(\Gamma)$ are conjugate to elements of $\Gamma$.
\item If $\Gamma$ is infinite, then $\mu(\Gamma)$ has the same
  cardinality as $\Gamma$, otherwise $\mu(\Gamma)$ is countably
  infinite.
\item The group~$\mu(\Gamma)$ contains a non-abelian free subgroup.
\item If $\Gamma$ is $n$-generated, then $\mu(\Gamma)$ is
  $(n+3)$-generated. In particular, if $\Gamma$ is finitely
  generated, then $\mu(\Gamma)$ is finitely generated.
\end{enumerate}
\end{intro_thm}

Thanks to parts~(1) and~(4) of the theorem, in the same spirit as in classical
embedding results~\cite{HNN}, we deduce:

\begin{intro_cor}[Corollary~\ref{cor:bac:cont}]\label{cor:intro:bac:cont}
  There exist continuum many non-isomorphic $5$-generated non-amenable
  boundedly acyclic groups.
\end{intro_cor}

Furthermore, we construct a finitely presented non-amenable boundedly acyclic group:

\begin{intro_thm}[A finitely presented non-amenable boundedly acyclic group; Corollary~\ref{cor:fp}]\label{thm:fp:bdd:acylic}
  There exists a finitely presented non-amenable boundedly acyclic group.
\end{intro_thm}

This group is non-amenable in a very strong sense, since it contains an isomorphic copy of every finitely presented group.

Our proofs are based on constructions with mitotic groups by
Baumslag--Dyer--Heller~\cite{BDH}
and Baumslag--Dyer--Miller~\cite{BDM}. As mitotic groups are far from
being finitely generated, we proceed as follows:
\begin{itemize}
\item Apply HNN-extensions to obtain finitely generated or finitely
  presented groups from mitotic groups;
\item Preserve bounded acyclicity along the construction.
\end{itemize}
This can be achieved by combining Monod--Popa's result on ascending
HNN-extensions~\cite{coamenable} with appropriate algebraic
constructions.

\subsection*{A note from the future}

After the first version of this article was posted,
new computations of bounded cohomology have emerged. 
It is now known that the bounded cohomology of~$\Homeo_+(S^1)$
is a polynomial ring, generated by the Euler class~\cite{monodnariman};
similarly, the bounded cohomology of Thompson's group~$T$ is a
polynomial ring, generated by the Euler class~\cite{binate}, provided
that Thompson's group $F$ is boundedly acyclic.

A proof of bounded acyclicity of~$F$ was recently given by Monod~\cite{monod_F}.
Along the way, he also shows that wreath products of the form $\Gamma \wr \mathbb{Z}$ are boundedly acyclic, for every group $\Gamma$.
This provides an alternative proof of the fact that every finitely generated group embeds into a finitely generated boundedly acyclic group.

%%%%%%%%%
\subsection{Finitely generated groups with large bounded cohomology}

Conversely, it is also interesting to construct finitely generated 
groups \emph{with large bounded cohomology}:

\begin{intro_defi}[Groups with large bounded cohomology]\label{defi:groups:with:large:bounded}
  A group $\Gamma$ has \emph{large bounded cohomology} if $\dim_\R
  \HH^n_b(\Gamma; \R) \geq |\R|$ for all $n \geq 2$.
\end{intro_defi}

Countable examples with large bounded cohomology can be constructed
through product constructions~\cite{Loeh}. Recently, Nitsche proved
that certain groups of homeomorphisms have similar
properties~\cite{Nitsche}.

Until now, no finitely generated examples with large bounded
cohomology were known~\cite{Monod:inv,OH2,FPS}. As recently remarked
by Heuer~\cite{Heuer_thesis}, ``It is notoriously hard to explicitly
compute bounded cohomology, even for most basic groups: There is no
finitely generated group $G$ for which the full bounded cohomology
$(\HH^n_b(G; \R))_{n \in \, \N}$ with real coefficients is known
except where it is known to vanish in all degrees".

\begin{intro_thm}[Finitely generated groups with large bounded cohomology (Corollary~\ref{cor:large:cont})]\label{thm:fg:large:bdd}
  There exist continuum many non-isomorphic $8$-generated groups with large
  bounded cohomology.
\end{intro_thm}

The main challenge in the proof of Theorem~\ref{thm:fg:large:bdd} is
to find a finitely generated group whose bounded cohomology is large
in infinitely many degrees (e.g., all even degrees): the rest can be
done by taking appropriate cross products.  Our construction starts
with a finitely generated group~$\Gamma$ introduced by
Meier~\cite{Meier} with the striking property of being isomorphic to
its direct square.  We then introduce a sufficient condition for
groups with this property to have large bounded cohomology in all even
degrees (Theorem~\ref{thm:large2}) and we show that $\Gamma$ satisfies
that hypothesis.

It is not clear whether this construction can produce finitely
presented examples (Scholium~\ref{scholium:fp}). But a concrete
finitely presented example can be obtained through Thompson's
group~$T$:

\begin{intro_thm}[A finitely presented group with large bounded cohomology
    (Theorem~\ref{thm:largefp})]\label{thm:fp:large:bdd}
  If $\Lambda$ is the fundamental group of an oriented closed
  connected hyperbolic $3$-manifold, then $T \times \Lambda$
  has large bounded cohomology.
\end{intro_thm}

%%%%%%%%%
\subsection{Non-computability}

We consider the problem of algorithmic computability of bounded
cohomology.  Despite of the Hopf formula for group homology, the
algorithmic problem
  \begin{itemize}
    \item[] Given a finite presentation~$\genrel SR$, decide
      whether~$\HH_2(\genrel SR;\Z)$ is trivial or not.
  \end{itemize}
  is undecidable, as one can show by a variation of the Adian--Rabin
  constructions~\cite[Theorem~4]{gordon}. The same method also shows
  that the algorithmic problem
  \begin{itemize}
    \item[] Given a finite presentation~$\genrel SR$, decide
      whether~$\HH^2(\genrel SR;\R)$ is trivial or not.
  \end{itemize}
  is undecidable.  Similarly, we obtain for bounded cohomology:

\begin{intro_thm}[Non-computability;
    Theorem~\ref{thm:noncomp}]\label{ithm:noncomp}
  Let $d \in \N_{\geq 2}$.  The following algorithmic
  problems are undecidable: Given a finite presentation~$\genrel SR$,
  decide whether
  \begin{enumerate}
  \item $\HH_b^{d}(\genrel SR;\R) \cong 0$ or not;
  \item $\dim_\R \HH_b^{d}(\genrel SR ;\R) = |\R|$ or not;
  \item $\genrel SR$ is boundedly acyclic or not.
  \end{enumerate}
\end{intro_thm}

Further statements of this type are contained in
Theorem~\ref{thm:noncomp}.
The witness constructions used in the proof of Theorem~\ref{ithm:noncomp} also apply to show non-computability results for $L^2$-Betti numbers and cost (Remark~\ref{rem:compl2}).

Ordinary cohomology of finite simplicial complexes with coefficients
in $\R$ or~$\Z$ is computable through elementary algorithms from
linear algebra. In contrast, Gromov's Mapping Theorem lets us deduce
from Theorem~\ref{ithm:noncomp} that the corresponding property does
\emph{not} hold for bounded cohomology:

\begin{intro_thm}[Non-computability for spaces;
    Corollary~\ref{cor:noncompspaces}]\label{ithm:noncompspaces}
  Let $d \in \N_{\geq 2}$.
  The following algorithmic problems are undecidable: Given a
  finite simplicial complex~$X$, decide whether
  \begin{enumerate}
  \item $\HH_b^{d}(X;\R) \cong 0$ or not;
  \item $\dim_\R \HH_b^{d}(X;\R) = |\R|$ or not;
  \item $X$ is boundedly acyclic or not.
  \end{enumerate}
\end{intro_thm}

Undecidability of vanishing in degrees~$\geq 5$ could also be deduced
from Weinberger's non-computability result for simplicial
volume~\cite[Chapter~2.6]{weinberger}. Our proof is based on similar
witness constructions.

We conclude by asking:

\begin{intro_quest}
  Which sequences of semi-normed vector spaces can be realised as
  bounded cohomology~$\HH_b^*(\Gamma;\R)$ of finitely
  generated/finitely presented groups~$\Gamma$?
\end{intro_quest}

%%%%%
\subsection*{Acknowledgements}

We are grateful to James Farre, Stefan Friedl, Ro\-berto Frigerio, Yash Lodha, Nicolas Monod and George Raptis for some useful conversations.

%%%%%%
\subsection*{Organisation of this article}

In Section~\ref{sec:prelim}, we collect background material: We recall
the basics of bounded cohomology and $\ell^1$-homology with some of
their properties. In Section~\ref{subsec:embedding:classical} we list
classical embedding theorems that we will need in the sequel.  In
Section~\ref{subsec:mitotic}, we define mitotic groups and discuss some
classical constructions.

Section~\ref{sec:boundedly:acyclic:properties} is devoted to the study
of some closure properties of the class of boundedly acyclic groups.  In
Section~\ref{sec:finitely:gen}, we construct finitely generated
boundedly acyclic groups; in particular, we prove
Theorem~\ref{thm:fg:bdd:acylic}.  In Section~\ref{sec:finitely:pres},
we construct a finitely presented non-amenable boundedly acyclic group, proving
Theorem~\ref{thm:fp:bdd:acylic}.  Section~\ref{sec:large} contains the
construction of groups with large bounded cohomology and the proof of
Theorem~\ref{thm:fg:large:bdd}. The proof of Theorem~\ref{thm:fp:large:bdd}
is given in Section~\ref{sec:largefp}.  Finally, we prove the
non-computability results Theorem~\ref{ithm:noncomp} and
Theorem~\ref{ithm:noncompspaces} in Section~\ref{sec:noncomp}. The
appendix contains basics on (co)homological dimension in the setting
of bounded cohomology, which are used in Section~\ref{sec:noncomp}.

%%%%%%

%%%%%%%%%%%%%%%%%%%%%%%%%%%%%%%%%%%%%%%%%%%%%%%%%%%%%%%%%%%
\section{Basic definitions}\label{sec:prelim}

Unless explicitly stated, all groups considered are discrete.

%%%%%%%%
\subsection{$\ell^1$-Homology and bounded cohomology}\label{subsec:bdd:cohomology}

We recall the notions of $\ell^1$-\emph{homology} and \emph{bounded
  cohomology} of groups.  Bounded cohomology of groups was introduced by
Johnson~\cite{Johnson} and Trauber and then extended to spaces by
Gromov~\cite{vbc}.  We recall the definition of bounded cohomology of
groups.  For convenience, we introduce bounded cohomology as a dual
construction to $\ell^1$-homology~\cite{loehl1}.

\subsubsection{$\ell^1$-Homology}

For a group~$\Gamma$, let $\CC_\bullet(\Gamma)$ denote the simplicial
resolution of~$\Gamma$ over~$\R$. For~$n \in \N = \{0,1,\dots\}$, we
have~$\CC_n(\Gamma)
\coloneqq \bigoplus_{g \in \Gamma^{n+1}} \R \cdot g$ and
\begin{align*}
  \partial_n \colon \CC_n(\Gamma) &\longrightarrow \CC_{n-1}(\Gamma)
  \\
  (g_0,\dots, g_n) & \longmapsto \sum_{j=0}^n (-1)^n \cdot (g_0,\dots, \widehat g_j, \dots, g_n).
\end{align*}
We endow $\CC_\bullet(\Gamma; \R)$ with the $\ell^1$-\emph{norm}:
$$
\Biggl|\sum_{g \in \, \Gamma^{\bullet+1}} a_{g} \cdot g \Biggr|_1 := \sum_{g \in \, \Gamma^{\bullet+1}} |a_g|.
$$
The boundary operator~$\partial_\bullet$ is bounded in each degree;
thus, we can define the \emph{$\ell^1$-resolution}~$\lonech \bullet
(\Gamma)$ of~$\Gamma$ as the completion of~$\CC_\bullet(\Gamma)$ with
respect to the $\ell^1$-norm.

\begin{defi}[$\ell^1$-Homology of groups]
  Let $\Gamma$ be a group and let $V$ be a Banach $\Gamma$-module
  (e.g., $\R$ with the trivial $\Gamma$-action).
  We set
  \[ \lonech \bullet (\Gamma;V) := \lonech \bullet(\Gamma) \mathbin{\overline\otimes_\Gamma} V,
  \]
  where $\overline\otimes$ denotes the projective tensor product.
  Then the \emph{$\ell^1$-homology} of~$\Gamma$ with coefficients
  in~$V$ is defined by
  \[
  \lonehom \bullet(\Gamma;V) := 
  \HH_\bullet\bigl( \lonech \bullet(\Gamma;V)\bigr).
  \]
\end{defi}

We will mainly be interested in the case of trivial $\R$-coefficients. 
The construction of~$\lonehom \bullet (\args;\R)$ is functorial
with respect to group homomorphisms.

\begin{rem}
  We recall that $\ell^1$-homology groups are endowed with an
  $\ell^1$-seminorm induced by the $\ell^1$-norm on
  $\lonech \bullet(\Gamma; \R)$.  In the sequel
  we will also make use of the \emph{reduced 
    $\ell^1$-homology} $\rlonehom \bullet(\Gamma; \R)$:
  $$
  \rlonehom\bullet(\Gamma;\R) := 
  \frac{\ker \big(\partial_\bullet \colon \lonech \bullet(\Gamma; \R) \to \lonech {\bullet-1}(\Gamma; \R)\big)}%
       {\text{$|\cdot|_1$-closure of }{\partial_{\bullet+1} \lonech {\bullet+1}(\Gamma; \R)}}.
  $$
\end{rem}

\subsubsection{Bounded cohomology}

The construction of $\ell^1$-homology allows us to define bounded
cohomology as follows~\cite{loehthesis}:

\begin{defi}[Bounded cohomology]
  Let $\Gamma$ be a group and let $V$ be a Banach $\Gamma$-module.
  The \emph{bounded cochain complex} of~$\Gamma$ with coefficients
  in~$V$ is defined as the $\Gamma$-invariants of the topological
  dual:
  \[ \CC^\bullet_b(\Gamma;V) := B\bigl( \lonech \bullet(\Gamma) , V \bigr)^\Gamma.
  \]
  The \emph{bounded cohomology of~$\Gamma$ with coefficients in~$V$} is
  then defined as
  \[ \HH^\bullet_b(\Gamma;V) := \HH^\bullet\bigl( \CC^\bullet_b(\Gamma;V)\bigr).
  \]
\end{defi}

The construction of~$\HH^\bullet_b(\args;\R)$ is functorial with respect
to group homomorphisms.

\begin{rem}
  Bounded cohomology with trivial real coefficients is well known in
  degrees $0$ and~$1$: For every group $\Gamma$ we have
  $\HH^0_b(\Gamma; \R) \cong \R$ and $\HH^1_b(\Gamma; \R) \cong
  0$~\cite{vbc,Frigerio}.  For this reason, we did not specify low degrees
  in Definitions~\ref{def:boundedly:acyclic:groups}
  and~\ref{defi:groups:with:large:bounded}.
\end{rem}

\subsubsection{Duality with $\ell^1$-homology}

We recall how to use reduced $\ell^1$-homology and the interaction
through the evaluation map~$\HH_b^\bullet \otimes_\R \rlonehom \bullet
\longrightarrow \R $ to show non-vanishing of bounded cohomology:

\begin{defi}
  Let $\Gamma$ be a group and let $k \in \N$. We set
  \[ \loneb k (\Gamma) := \dim_\R \rlonehom k (\Gamma;\R). %\in \N \cup \{\infty\}.
  \]
\end{defi}

\begin{prop}[\protect{\cite[Proposition~3.2, Proposition~3.3]{Loeh}}]
  \label{prop:loneb}
  Let $\Gamma$ and $\Lambda$ be groups and let $k,m \in \N$. Then, we have:
  \begin{enumerate}
  \item $\dim_\R \HH_b^k(\Gamma;\R) \geq \loneb k(\Gamma)$;
  \item $\loneb 2(\Gamma) \geq \dim_\R \HH_b^2(\Gamma;\R)$~\cite[Theorem~2.3, Corollary~2.7]{MM};
  %\item $\loneb k (\Gamma \times \Lambda) \geq \loneb k (\Gamma)$;
  \item $\loneb {k+m} (\Gamma \times \Lambda) \geq \loneb k (\Gamma) \cdot \loneb m (\Lambda)$;
  \end{enumerate}
\end{prop}

\begin{rem}
\label{rem:loneb}
  The previous proposition has the following special situations:
  \begin{itemize}
  \item If $\Gamma_d = \Lambda \times \Lambda_2^d$ for some group $\Lambda_2$, then
    $$\dim_\R \HH_b^{2d}(\Gamma_d;\R) \geq \loneb {2d} (\Gamma_d) \geq \loneb 2 (\Lambda_2)^d.$$
  \item For every group $\Lambda_3$, the above property in degree~$2d$ still holds for~$\Gamma_d \times \Lambda_3$, and moreover
    $$\dim_\R \HH_b^{2d+3}(\Gamma_d \times \Lambda_3;\R) \geq \loneb {2d + 3} (\Gamma_d \times \Lambda_3)
    \geq \loneb {2d} (\Gamma_d) \cdot \loneb {3} (\Lambda_3).$$
  \end{itemize}
  For $k \in\{ 2, 3\}$, choosing $\Lambda_k$ with $\loneb k(\Lambda_k) > 0$, we thus
  obtain lower bounds on the dimension of the bounded cohomology
  spaces.  This will be used in Sections~\ref{sec:large} and~\ref{sec:largefp} to construct
  groups with large bounded cohomology, and in
  Section~\ref{sec:noncomp} to extend results in degree $2$ and~$3$ to
  higher degrees.
\end{rem}

\begin{example}\label{exa:hyp}
  If $\Gamma$ is the fundamental group of an oriented closed connected
  hyperbolic $n$-manifold~$M$, then $\loneb n (\Gamma) > 0$,
  because the simplicial volume of~$M$ is non-zero~\cite{vbc,Thurston}
  and hence the fundamental class of~$M$ yields a non-trivial
  element of~$\rlonehom n (M;\R) \cong \rlonehom n (\Gamma;\R)$.
\end{example}

%%%%%%%%

\subsection{Embedding theorems}\label{subsec:embedding:classical}

We collect some classical embedding theorems, which we will need in
the sequel (Sections~\ref{sec:finitely:gen}, \ref{sec:finitely:pres}
and~\ref{sec:large}).  We will adopt the following notation for group
presentations: Given a group presentation $H = \langle S \mid R
\rangle$ and a disjoint set of generators $S'$, we will simply write
$$
\langle H; S' \mid R' \rangle \coloneqq \langle S \cup S' \mid R \cup R' \rangle,
$$
where $R'$ is a new set of relations (possibly) involving both
elements in $S$ and $S'$. Moreover, it will always be clear from
the context whether $\genrel SR$ denotes a presentation or the
group given by this presentation.

\begin{defi}[Recursively presented group]\label{def:rec:pres:groups}
  A group presentation $\langle S \mid R \rangle$ is \emph{recursively
    enumerable} if the generating set $S$ is countable and the set of
  relations $R$ is a recursively enumerable subset of the free group
  over $S$.  A group is \emph{recursively presented} if it has a
  recursively enumerable presentation.
\end{defi}

The importance of recursively presented groups is elucidated by the
following well-known result by Higman:

\begin{thm}[{\cite[Theorem~1]{Higman}}]
  \label{thm:Higman}
  A group is recursively presented if and only if it embeds into a
  finitely presented group.
\end{thm}

Moreover, Higman also showed that there exists a universal finitely
presented group in the following sense:

\begin{thm}[{\cite[p.~456]{Higman}}]
  \label{thm:Higman:uni}
  There exists a universal finitely presented group, that is, 
  a finitely presented group that contains
  an isomorphic copy of every finitely presented group.
\end{thm}

To move from one to continuum many examples in
Corollary~\ref{cor:intro:bac:cont} and Theorem~\ref{thm:fg:large:bdd},
we make use of the following:

\begin{thm}[{\cite[Section~4]{HNN}}]
  \label{thm:HNN}
  Every countable group $\Gamma$ embeds into a $2$-generator group~$K$
  with the following property: For every prime~$p$ the group~$K$
  contains $p$-torsion if and only if $\Gamma$ does.
\end{thm}

\begin{cor}
\label{cor:HNN:cont}
There exist continuum many $2$-generator groups, which are pairwise
distinguished by their torsion.
\end{cor}

\begin{proof}
  For each (possibly infinite) set $P$ of prime numbers, it is
  sufficient to apply Thorem~\ref{thm:HNN} to
  $\bigoplus_{p \in P} \Z / p \Z$.
\end{proof}

\begin{rem}
\label{rem:continuum:tf}

More explicit constructions allow to show that there exist continuum many pairwise non-isomorphic $2$-generator groups, which are moreover torsion-free~\cite{Camm}.
\end{rem}

\begin{rem}\label{rem:HNN:embedding}
  Notice that all these results make essential use of
  HNN-extensions. Indeed, every group $\Gamma$ embeds into every 
  HNN-extension of the form~$\Gamma \ast_\varphi$~\cite{HNN}.  We will
  use this fact repeatedly in the sequel.
\end{rem}

%%%%%%%%%%

\subsection{Mitotic groups}\label{subsec:mitotic}

\emph{Mitotic groups} are acyclic groups, first introduced by
Baumslag, Dyer and Heller~\cite{BDH} as building blocks in order to
prove new results about functorial embeddings of groups into acyclic
groups. Mitotic groups are based on \emph{mitoses}:

\begin{defi}[Mitosis]\label{def:mitosis}
  Let $H$ be a subgroup of a group $\Gamma$. We say that $\Gamma$ is a
  \emph{mitosis of} $H$ if there exist two elements $s, d \in \,
  \Gamma$ such that the following hold
  \begin{enumerate}
  \item $\Gamma$ is generated by $H$, $s$ and $d$;
  \item For all~$h, h' \in \, H$, we have $[h', s^{-1} h s] = 1$.
  \item For all~$h \in \, H$, we have 
    $
    d^{-1} h d = h s^{-1} h s;
    $
  \end{enumerate}
  We use the commutator convention~$[x,y]:=x^{-1}y^{-1}xy$.
\end{defi}

The second condition says that $\Gamma$ contains two
conjugate, commuting copies of~$H$
(which are allowed to intersect non-trivially).
This implies that there exists a
third, diagonal copy, and the third condition says that this is also
conjugate to the first copy.  This is better illustrated by the
next example: Following Baumslag, Dyer and Heller~\cite[Section
  5]{BDH}, we recall how to embed every group into its
\emph{standard mitosis}:

\begin{example}[Standard mitosis]\label{ex:standard:mitosis}
  Let $\Gamma$ be a group. We begin by embedding $\Gamma$
  into $\Gamma \times \Gamma$ via the identification $g \mapsto (g,
  1)$. Then, we construct a mitotis for $\Gamma$ by applying two
  HNN-extensions as follows: We add an element $s$ that conjugates
  $(g, 1)$ to $(1, g)$, and then an element $d$ that conjugates $(g,
  1)$ to $(g, g)$.  This leads to a group
  \begin{align*}
    m(\Gamma) &\coloneqq \langle \Gamma \times \Gamma; s, d \mid s^{-1} (g, 1) s = (1, g), d^{-1} (g, 1) d = (g, g) \rangle \\
    &= \langle \Gamma; s, d \mid d^{-1} g d = g s^{-1} g s, [g, s^{-1} h s] = 1 : g, h \in \Gamma \rangle,
  \end{align*}
  which is called the \emph{standard mitosis} of $\Gamma$. It is immediate
  to check that $\Gamma$ is embedded into $m(\Gamma)$
  (Remark~\ref{rem:HNN:embedding}).

  This construction is functorial in the following
  sense~\cite[Lemma~5.2]{BDH}: Every homomorphism $\varphi \colon
  \Gamma_1 \to \Gamma_2$ induces a homomorphism $m(\varphi) \colon
  m(\Gamma_1) \to m(\Gamma_2)$ just by sending $g \to \varphi(g)$ for
  every $g \in \, \Gamma_1$ and $s_1 \mapsto s_2$, $d_1 \mapsto d_2$.
  Moreover, the functor $m \colon \Groups \to \Groups$ preserves
  monomorphisms.
\end{example}

To prove non-amenability of the boundedly acyclic groups constructed
in Theorem~\ref{thm:fg:bdd:acylic}, the following will be useful:

\begin{rem}
\label{rem:mitosis:free}
  The subgroup of~$m(\Gamma)$ generated by~$s$ and~$d$ is free of
  rank~$2$: Indeed, the map $m(\Gamma) \to F_2$ sending $s$ and $d$ to
  the two generators and annihilating $\Gamma$ is a surjective
  homomorphism, as can be easily seen from the latter presentation of
  $m(\Gamma)$.
\end{rem}

\begin{defi}[Mitotic groups]\label{def:mitotic}
  A group $\Gamma$ is \emph{mitotic} if every finitely generated
  subgroup of~$\Gamma$ admits a mitotis in~$\Gamma$.
\end{defi}

Using the construction of Example~\ref{ex:standard:mitosis}, it is
easy to show that every group embeds into a mitotic
group~\cite[Lemma~5.4]{BDH}:

\begin{example}[Standard mitotic embedding]\label{ex:standard:mitotic:embedding}
  Given a group $\Gamma$ we already know that it can be
  embedded into its standard mitosis $m(\Gamma)$
  (Example~\ref{ex:standard:mitosis}).  Moreover, since the functor~$m$
  preserves monomorphisms, we can construct the direct union
  (colimit) $m^\infty(\Gamma) = \bigcup_{i \geq 0} m^i(\Gamma)$, which
  is mitotic.
\end{example}

Mitotic groups are known to be acyclic~\cite{BDH}.  Our
interest in them is motivated by the following result (see
Definition~\ref{def:boundedly:acyclic:groups}):

\begin{thm}[{\cite[Theorem~1.2]{Loeh}}]
  \label{thm:mitotic:bac}
  All mitotic groups are boundedly acyclic.
\end{thm}

%%%%

\section{Basic properties of boundedly acyclic groups}\label{sec:boundedly:acyclic:properties}

The class~$\bac$ of \emph{boundedly acyclic groups} consists of those
groups with vanishing bounded cohomology in all positive degrees and
trivial real coefficients $\R$
(Definition~\ref{def:boundedly:acyclic:groups}).  They appeared in the
literature also under the name of groups with zero bounded
cohomological dimension~\cite{Loeh}, but we prefer to stick with the
name boundedly acyclic because of the recent characterisation of
boundedly acyclic maps~\cite{BAc} and Ivanov's work on boundedly
acyclic open covers~\cite{Ivanov_bac_covers}. Moreover, this choice
also avoids any confusion with the bounded cohomological
dimension~$\cdb$ that we will introduce and discuss later
(Sections~\ref{sec:noncomp} and Appendix~\ref{appx:cdb}).  In this
section we discuss some hereditary properties of boundedly acyclic
groups.

%%%%%%%%%%%%%%%%%%

\subsection{Boundedly acyclic groups and HNN extensions}

We recall the definition of \emph{co-amenability} and
how this property provides information about~$\bac$
and HHN-extensions.

\begin{prop}[{\cite[Corollary 4.2.2]{BAc}}]
  \label{prop:ext}
  Let $1 \to H \to \Gamma \to Q \to 1$ be a short exact sequence of
  groups. Suppose $H \in \bac$. Then $\Gamma \in \bac$ if and only if
  $Q \in \bac$.
\end{prop}

In the situation of Proposition~\ref{prop:ext},
if $H \in \bac$ and $Q$ is amenable, then $\Gamma \in
\bac$. In fact, following a result by Monod and
Popa~\cite{coamenable}, one can strengthen this statement
via co-amenability:

\begin{defi}[Co-amenable subgroups]\label{def:coamenable}
  Let $\Gamma$ be a group and let $H \leq \Gamma$ be a subgroup. We
  say that $H$ is \emph{co-amenable} in $\Gamma$ if there exists a
  $\Gamma$-invariant mean on the space~$\ell^\infty(\Gamma\slash H)$
  of bounded functions on~$\Gamma \slash H$.
\end{defi}

\begin{example}\label{ex:coamenable:subgroups}
  The following are examples of co-amenable subgroups:
  \begin{enumerate}
  \item Suppose that $H$ is normal in $\Gamma$.  Then
    $H$ is co-amenable in~$\Gamma$ if and only if the
    quotient~$\Gamma \slash H$ is amenable.
  \item Let $H$ be a group and let $\varphi \colon H \to H$ a
    monomorphism.  Let $\Gamma = H \ast_\varphi$ be the corresponding
    HNN-extension: Following most of the literature, we will call
    such HNN-extensions \emph{ascending}. Then, $H$ is co-amenable in
    $\Gamma$~\cite[Proposition~2]{coamenable}.
  \item Given a chain~$K < H < \Gamma$ of groups, we have: If $K$
    is co-amenable in~$H$ and $H$ is co-amenable in~$\Gamma$, then $K$
    is also co-amenable in~$\Gamma$~\cite{coamenable}.
  \item Given a chain~$K < H < \Gamma$ of groups, we have: If $K$
    is co-amenable in $\Gamma$, then $H$ is co-amenable in $\Gamma$~\cite{coamenable}.
    \item On the other hand, it is not true in general that
    the co-amenability of~$K$ in~$\Gamma$ implies that $K$ is
    co-amenable in $H$~\cite{coamenable}.
\end{enumerate}
\end{example}

The importance of co-amenability in our setting is evident from the
following result by Monod and Popa:

\begin{prop}[{\cite[Proposition~3]{coamenable}}, {\cite[8.6]{monod}}]
  \label{prop:coam}
  Let $H \leq \Gamma$ be a co-amenable subgroup. Then, the inclusion
  map $i \colon H \to \Gamma$ induces an injective map in bounded
  cohomology:
  $$
  \HH^\bullet_b(i) \colon \HH^\bullet_b(\Gamma; \R) \hookrightarrow \HH^\bullet_b(H; \R)
  $$
  In particular, if $H \in \bac$, then $\Gamma \in \bac$.
\end{prop}

We will see in Corollary~\ref{cor:HNN:nonsurj} that $\HH^\bullet_b(i)$
is, in general, far from being surjective.

\begin{cor}\label{cor:HNNbac}
  Ascending HNN-extensions of groups in~$\bac$ are in~$\bac$.
\end{cor}
\begin{proof}
  We combine Proposition~\ref{prop:coam} and
  Example~\ref{ex:coamenable:subgroups}.2.
\end{proof}
  
The following example shows that Corollary~\ref{cor:HNNbac} 
does not hold for general HNN-extensions.

\begin{example}
  The group $\BS(2, 3)$ is an HNN-extension of $\Z$ along finite-index
  subgroups isomorphic to $\Z$.  However, $\BS(2, 3)$ is not a
  boundedly acyclic group. Indeed, it admits non-trivial
  quasimorphisms~\cite{Grigorchuk:free, Fujiwara:free}, whence it has
  non-trivial second bounded cohomology group.
\end{example}

\subsection{Quotients and $\lex$ groups}
\label{subs:quot}

One of the most remarkable properties of mitotic groups is that they
are closed under quotients~\cite[Appendix~B]{BerrickDic}.  Since the
same property is also true for amenable
groups~\cite[Proposition~3.4]{Frigerio}, a natural question is whether
the quotient of a boundedly acyclic group is still boundedly acyclic.

This problem is related to the problem of showing existence of a
non-$\lex$ group, i.e., of groups that are not left-exact in the
following sense:

\begin{defi}[$\lex$ groups~\cite{Bouarich}]
  We say that a group $\Gamma$ lies in the family $\lex$ if it
  satisfies the following left-exactness property: For every
  group~$\Lambda$ and every epimorphism $\psi \colon \Lambda \to
  \Gamma$, the induced map~$\HH^\bullet_b(\psi)$ in bounded cohomology
  with trivial real coefficients is injective in all degrees.
\end{defi}

It is an open problem to find examples of groups that do not lie in
$\lex$.  On the other hand, in degree $2$ the situation is understood:
Every group epimorphism induces an injective map between the second
bounded cohomology groups~\cite{Bouarich2} (in fact the induced map is
even isometric~\cite[Theorem~2.14]{Huber} and the result also holds
for a larger family of coefficients~\cite[Example~4.1.2]{BAc}).

Of course amenable groups and, more generally, boundedly acyclic groups
lie in $\lex$. Some more interesting examples are for
instance:
\begin{itemize}
\item Free groups;
\item Fuchsian groups~\cite[Corollaire~3.9]{Bouarich};
\item Fundamental groups of geometric
  $3$-manifolds~\cite[Corollaire~3.13 and p.~267]{Bouarich}
  (together with Agol's proof of Thurston's Virtual Fibering
  Conjecture~\cite{Agol}).
\item The class~$\lex$ is closed under the following constructions:
  Quotients by amenable subgroups, extensions of an amenable group by
  an element in $\lex$~\cite[Proposition~3.16]{Bouarich} and free
  products of amenable groups amalgamated over a common normal
  subgroup~\cite[Corollaire~3.17]{Bouarich}.
\end{itemize} 

\begin{rem}
  Following Bouarich's proof~\cite[Proposition~3.16]{Bouarich}, one
  can in fact deduce the corresponding statement for boundedly acyclic
  groups.  Namely, $\lex$ is also closed under the following
  constructions:
  \begin{itemize}
  \item Quotients by boundedly acyclic groups;
  \item Extensions of a boundedly acyclic group by an element in $\lex$.
  \end{itemize}
  The proof follows \emph{verbatim} the one by Bouarich with the
  additional fact that epimorphisms with boundedly acyclic kernels
  induce isomorphism in all bounded cohomology groups with trivial
  real coefficients~\cite[Theorem~4.1.1]{BAc}
\end{rem}

The connection between quotients of $\bac$ groups and $\lex$ groups is
given in the following proposition:

\begin{prop}
  \label{prop:bac:quot}
  The family $\bac$ is closed under quotients if and only if all 
  quotients of boundedly acyclic groups lie in $\lex$.
\end{prop}

\begin{proof}
  Since boundedly acyclic groups lie in $\lex$, one implication
  trivially holds.  Vice versa, if we assume that $\bac$ is not closed
  under quotients, there exists an epimorphism $\psi \colon
  \Gamma \to \Lambda$ with $\Gamma \in \, \bac$ and $\Lambda \not\in
  \bac$.  Then, $\psi$ cannot induce an injective map in bounded
  cohomology for all degrees.  This shows that $\Lambda \not\in \lex$,
  whence the thesis.
\end{proof}

This proposition provides a strategy to find a non-$\lex$ group:
namely, it would suffice to exhibit a quotient of a $\bac$ group that
is not in~$\bac$.  Note that by the discussion above, every quotient
of a $\bac$ group has vanishing second bounded cohomology with trivial
real coefficients.  Several groups of geometric nature are then
natural candidates for a counterexample:

\begin{example}
  Let $X$ be an $n$-dimensional irreducible symmetric space of
  non-compact type, and $G$ the associated Lie group. Let $\Gamma < G$
  be a torsion-free cocompact lattice.  Then $\| \Gamma \backslash X
  \| > 0$~\cite{Lafont_Schmidt}, and so by Gromov's Duality
  Principle~\cite{vbc} it follows that $\HH^n_b(\Gamma; \R) \not\cong
  0$ (in fact something can be said about lower degrees as
  well~\cite{Lafont_Wang}).

  On the other hand, if $X$ is not Hermitian symmetric and has real
  rank at least $3$, then $\HH^2_b(\Gamma; \R) \cong 0$~\cite{lattices}.
\end{example}

%%%%%%

\section{Finitely generated boundedly acyclic groups}\label{sec:finitely:gen}

In this section, we show that each finitely generated group embeds into
a finitely generated boundedly acyclic group. The latter will always
contain a non-abelian free group, providing the first examples of (an
infinite family of) non-amenable finitely generated boundedly acyclic
groups.

\begin{thm}
  \label{thm:fg}
  There exists a functor $\mu \colon \Groups \to \Groups$ associating
  to each group $\Gamma$ a boundedly acyclic group $\mu(\Gamma)$ into
  which $\Gamma$ embeds.  The group $\mu(\Gamma)$ has the following
  properties:
  \begin{enumerate}
  \item Torsion elements of $\mu(\Gamma)$ are conjugate to elements of $\Gamma$.
  \item If $\Gamma$ is infinite, then $\mu(\Gamma)$ has the same
    cardinality as $\Gamma$, otherwise $\mu(\Gamma)$ is countably
    infinite.
  \item The group~$\mu(\Gamma)$ contains a non-abelian free subgroup.
  \item If $\Gamma$ is $n$-generated, then $\mu(\Gamma)$ is
    $(n+3)$-generated. In particular, if $\Gamma$ is finitely generated,
    then $\mu(\Gamma)$ is finitely generated.
\end{enumerate}
\end{thm}

\begin{proof}
  We construct our functor starting with the standard mitosis of
  $\Gamma$ (Example~\ref{ex:standard:mitosis}): Every group $\Gamma$
  embeds into its standard mitosis
  $$m(\Gamma) = \langle \Gamma; s, d \mid d^{-1} g d = g s^{-1} g s, [g, s^{-1} h s] = 1 : g, h \in \Gamma \rangle.$$
  To make the notation more transparent, let us denote $s$ by $s_1$
  and $d$ by $d_1$.  Then, we iterate the process as in the standard
  mitotic embedding (Example~\ref{ex:standard:mitotic:embedding}):
  Denoting by $s_i$ and $t_i$ the new generators of $m^i(\Gamma)$, we
  obtain the directed union (colimit) $m^\infty(\Gamma)$, which is
  generated by~$\Gamma$, together with~$\{ s_1, d_1, \ldots \}$.
  There exists a self-monomorphism $\varphi$ of~$m^\infty(\Gamma)$ given by $g
  \mapsto g, s_i \mapsto s_{i+1}, d_i \mapsto d_{i+1}$~\cite[p.~20]{BDH}.
  We now set $\mu(\Gamma)$ to be the
  ascending HNN-extension $m^\infty(\Gamma) \ast_\varphi$:
  $$
  \mu(\Gamma) \coloneqq \langle m^\infty(\Gamma); t \mid t^{-1} x t = \varphi(x) : x \in m^\infty(\Gamma) \rangle.
  $$
  Notice that $\mu$ is in fact a functor, because it is constructed
  via iterated standard mitoses (Example~\ref{ex:standard:mitosis})
  followed by an HNN-extension. Moreover, $\mu(\Gamma)$ is a boundedly
  acyclic group because it is an ascending HNN-extension of the mitotic
  group $m^\infty(\Gamma)$ (Theorem~\ref{thm:mitotic:bac} and
  Corollary~\ref{cor:HNNbac}). Finally, by construction, $\Gamma$
  embeds into $\mu(\Gamma)$ (Remark~\ref{rem:HNN:embedding}).

  We are left to check that $\mu(\Gamma)$ satisfies the properties
  \emph{(1)--(4)}.

  \emph{Ad~1}. The statement on torsion follows from the fact that all
  torsion in an HNN-extension is conjugate into the base
  group~\cite{HNN}.

  \emph{Ad~2}. The statement on the cardinality follows from
  the fact that HNN-extensions of infinite groups preserve the
  cardinality, while HNN-extensions of finite groups are countably infinite.

  \emph{Ad~3}. By Remark~\ref{rem:mitosis:free}, the subgroup
  of~$m(\Gamma)$ generated by~$s_1$ and~$d_1$ is free of rank~$2$.
  Since $m(\Gamma)$ also embeds into $\mu(\Gamma)$
  (Remark~\ref{rem:HNN:embedding}), we also have that $\mu(\Gamma)$
  contains a non-abelian free group.

  \emph{Ad~4}. It is immediate to check that the generators of
  $\Gamma$, together with $s_1, d_1$ and $t$, suffice to generate
  $\mu(\Gamma)$. This shows that if $\Gamma$ is $n$-generated, then
  $\mu(\Gamma)$ is $(n+3)$-generated, whence the claim.
\end{proof}

\begin{rem}
  Notice that $\mu(\Gamma)$ is not acyclic. Indeed, the presentation
  of $\mu(\Gamma)$ shows that its abelianization is an infinite cyclic
  group, whence $\HH_1(\mu(\Gamma); \mathbb{Z}) \cong \mathbb{Z}$
  ~\cite[p.~20]{BDH}. Nevertheless, it is worth mentioning that
  $\HH_n(\mu(\Gamma); \mathbb{Z}) \cong 0$ for all $n > 1$, because 
  $m^\infty(\Gamma)$ is acyclic~\cite[Theorem~4.2; p.~20]{BDH}.
\end{rem}

\begin{cor}\label{cor:5:gen:and:tors}
  Every countable group $\Gamma$ embeds into a $5$-generated
  non-amenable boundedly acyclic group that has $p$-torsion if and
  only if $\Gamma$ does.
\end{cor}
\begin{proof}
  This is the
  combination of parts~(1), (3) and~(4) of Theorem~\ref{thm:fg} 
  together with Theorem~\ref{thm:HNN}.
\end{proof}

Similar embedding results for \emph{mixing} coefficients
have been obtained by Monod~\cite[Proposition~6.5]{monod_sarithmetic}.

Using Corollary~\ref{cor:5:gen:and:tors}, the same proof as
Corollary~\ref{cor:HNN:cont} gives:

\begin{cor}
  \label{cor:bac:cont}
  There exist continuum many $5$-generated non-amenable
  boundedly acyclic groups, which are pairwise distinguished by their torsion.
\end{cor}

\begin{rem}
\label{rem:continuum:tf:bac}

In fact, there exist continuum many \emph{torsion-free} $5$-generated non-amenable boundedly acyclic groups, which are pairwise non-isomorphic.
Indeed, there exist continuum many pairwise non-isomorphic $2$-generated torsion-free groups $(\Gamma_i)_{i \in I}$, by Remark~\ref{rem:continuum:tf}.
Now each group $\mu(\Gamma_i)$ is finitely generated, in particular it is countable, and so has only countably many finitely generated subgroups.
Therefore there must be continuum many distinct isomorphism types in the collection $(\mu(\Gamma_i))_{i \in I}$ as well.
Finally, they are torsion-free by Theorem~\ref{thm:fg}.
\end{rem}

One natural question is whether the same construction leads to
finitely \emph{presented} boundedly acyclic groups. Unfortunately, as
in the classical case of acyclic groups, this is never the
case~\cite[Theorem~5.6]{BDH}.  Let us recall the following
construction by Baumslag, Dyer and Heller~\cite[Section~5]{BDH}: Let
$A$ be a finitely presented torsion-free acyclic group with generators
$a, b, c$ (one such example can be found in~\cite[Section~3]{BDH}),
and let $A(\Gamma) \coloneqq \mu(\Gamma) *_{t = a} A$.  Then,
$A(\Gamma)$ has the following properties:
\begin{enumerate}
\item[(P1)] The group~$A(\Gamma)$ is never finitely presented~\cite[Theorem
  5.6]{BDH};
\item[(P2)] If $A(\Gamma)$ is finitely generated, then so is
  $\Gamma$~\cite[Lemma 5.7]{BDH};
\item[(P3)] If $\Gamma$ is recursively presented, then so is
  $A(\Gamma)$~\cite[p.~38]{BDM}.
\end{enumerate}
Using this group we can deduce the following additional
properties of $\mu(\Gamma)$:

\begin{prop}\label{prop:add:prop:mu:gamma}
  The functor $\mu \colon \Groups \to \Groups$ associating to each
  group $\Gamma$ the boundedly acyclic group $\mu(\Gamma)$ also
  satisfies the following properties:
  \begin{enumerate}
  \item The group~$\mu(\Gamma)$ is finitely generated if and only if $\Gamma$ is;
  \item The group~$\mu(\Gamma)$ is never finitely presented;
  \item If $\Gamma$ is recursively presented, then so is $\mu(\Gamma)$.
  \end{enumerate}
\end{prop}
\begin{proof}
  We can easily deduce all the properties from the ones of
  $A(\Gamma)$:

  \emph{Ad~1}. We have already shown in Theorem~\ref{thm:fg}.4 that
  if $\Gamma$ is finitely generated, then also $\mu(\Gamma)$ is
  finitely generated. Vice versa, if $\mu(\Gamma)$ is finitely
  generated, then so is $A(\Gamma)$. This shows that $\Gamma$ is
  also finitely generated by~(P2) of~$A(\Gamma)$.

  \emph{Ad~2}. By contradiction, assume that $\mu(\Gamma)$ is
  finitely presented.  Then also $A(\Gamma)$ is finitely presented
  being a free product of finitely presented groups amalgamated along
  a finitely generated subgroup.  This leads to a contradiction ((P1)
  of $A(\Gamma)$).

  \emph{Ad~3} If $\Gamma$ is recursively presented, we already know
  that $A(\Gamma)$ is too ((P3) of~$A(\Gamma)$). This implies that
  $\mu(\Gamma)$ is also recursively presented, as subgroups of
  recursively presented groups are recursively presented~\cite{Higman}.
\end{proof}

The previous result has an important consequence, which will be a
building block in our construction of a finitely
presented non-amenable boundedly acyclic group (Theorem~\ref{thm:fp:bdd:acylic}).
Let $U$ denote a universal finitely presented group
(Theorem~\ref{thm:Higman:uni}). Then we have the following:

\begin{cor}\label{cor:muG:embeds:into:U}
  Let $\Gamma$ be a finitely presented group. Then, $\mu(\Gamma)$
  embeds into~$U$.
\end{cor}
\begin{proof}
  By definition of $U$ we know that every finitely presented group
  embeds into~$U$.  So it is sufficient to show that we can embed
  $\mu(\Gamma)$ into a finitely presented group.  Since $\Gamma$ is
  finitely presented (whence recursively presented),
  Proposition~\ref{prop:add:prop:mu:gamma}.3 shows that $\mu(\Gamma)$ is
  also recursively presented.  Hence, by Theorem~\ref{thm:Higman}
  $\mu(\Gamma)$ embeds into a finitely presented group, which in turn
  embeds into $U$ (Theorem~\ref{thm:Higman:uni}).
\end{proof}

%%%%%%%%%%%%%

\section{A finitely presented non-amenable boundedly acyclic group}
\label{sec:finitely:pres}

The aim of this section is to make use of Theorem~\ref{thm:fg} in
order to construct a finitely presented boundedly acyclic group that
contains all finitely presented groups. In particular this provides
the first example of a finitely presented non-amenable boundedly
acyclic group.  The fundamental tool in the process is the following
result, which is based on the techniques by Baumslag, Dyer and
Miller~\cite[Section 4]{BDM}:

\begin{thm}
  \label{thm:fp}
  Let $\Gamma$ be a group such that $\mu(\Gamma)$ embeds into
  $\Gamma$. Then $\Gamma$ has a boundedly acyclic ascending
  HNN-extension.  In particular, a universal finitely presented group
  has a boundedly acyclic ascending HNN-extension.
\end{thm}

\begin{proof}
  Let $\Gamma'$ be an isomorphic copy of $\Gamma$ with $\mu(\Gamma)
  \leq \Gamma'$.  Let $f \colon \Gamma' \to \Gamma$ be such an
  isomorphism.  Then there exists a monomorphism $\varphi \colon
  \Gamma' \to \Gamma \leq \Gamma'$ obtained by composing $f$ with the
  inclusion of $\Gamma$ into $\mu(\Gamma) \leq \Gamma'$
  (Theorem~\ref{thm:fg}).  We claim that the corresponding ascending
  HNN-extension~$\Gamma'*_\varphi$ is boundedly acyclic.  Indeed,
  since $\varphi(\Gamma') = t^{-1} \Gamma' t$ inside $\Gamma'
  \ast_\varphi$, we have the following chain of inclusions:
  $$
  \Gamma' = t \Gamma t^{-1} < t \mu(\Gamma) t^{-1} < t \Gamma' t^{-1} < \Gamma' \ast_\varphi.
  $$
  Moreover, $\Gamma'$ is co-amenable in $\Gamma'*_\varphi$
  (Example~\ref{ex:coamenable:subgroups}.2) and so by
  Example~\ref{ex:coamenable:subgroups}.4 we also know that $t
  \mu(\Gamma) t^{-1}$ is co-amenable in $\Gamma' *_\varphi$. By using
  the fact that $t \mu(\Gamma) t^{-1}$ is isomorphic to $\mu(\Gamma)$,
  whence boundedly acyclic, we conclude that $\Gamma' *_\varphi$ is
  boundedly acyclic as claimed (Proposition~\ref{prop:coam}).
  
  A universal finitely presented group (Theorem~\ref{thm:Higman:uni})
  satisfies the hypothesis of the theorem.  Indeed by
  Theorem~\ref{thm:fg} and Corollary~\ref{cor:muG:embeds:into:U}, we
  have embeddings $U < \mu(U) < U$.  Therefore $U$ admits a boundedly
  acyclic ascending HNN-extension.
\end{proof}

\begin{cor}
  \label{cor:fp}
  There exists a finitely presented boundedly acyclic group that
  contains an isomorphic copy of every finitely presented group.
\end{cor}

\begin{proof}
  Let $U$ be a universal finitely presented group.  It admits a
  boundedly acyclic ascending HNN-extension, by the last statement of
  Theorem~\ref{thm:fp}.  This is finitely presented, being an
  ascending HNN-extension of a finitely presented group, and it
  contains $U$, thus all finitely presented groups.
\end{proof}

By Proposition~\ref{prop:coam}, if $\Lambda$ is an
ascending HNN-extension of $\Gamma$, then $\Gamma$ is co-amenable in
$\Lambda$ and so the embedding $\Gamma \to \Lambda$ induces an
injection in bounded cohomology.  The techniques from this section
allow us to show that, in general, this injection is very far from
being an isomorphism:

\begin{cor}
\label{cor:HNN:nonsurj}
  For all $d \geq 2$ there exists a finitely presented group $\Gamma$
  with $\HH^d_b(\Gamma; \R) \not\cong 0$ that admits a boundedly acyclic
  ascending HNN-extension.
\end{cor}

\begin{proof}
  For every $d \geq 2$, let $\Lambda_d$ be the fundamental group of an
  oriented closed connected hyperbolic $d$-manifold.  Then, we know
  that $\HH^d_b(\Lambda_d; \R) \not\cong 0$ (Proposition~\ref{prop:loneb} and Example~\ref{exa:hyp})). 
  Since $\Lambda_d$ is
  finitely presented, we can set $\Gamma \coloneqq U \times
  \Lambda_d$, where $U$ is a universal finitely presented group.  As
  $\Gamma$ retracts onto~$\Lambda_d$, we have $\HH^d_b(\Gamma; \R)
  \not\cong 0$.  On the other hand, $\Gamma$ is itself a universal
  finitely presented group, so by the last statement of
  Theorem~\ref{thm:fp} it has a boundedly acyclic ascending
  HNN-extension.
\end{proof}

%\begin{rem}
%  The following question~\cite[Question 1.8]{Loeh} remains open: What
%  can be said about the bounded cohomology of the wreath product $F
%  \wr \mathbb{Z}$, where $F$ is a non-abelian free group?  On the one
%  hand, a direct sum of infinitely many copies of $F$ has
%  non-vanishing bounded cohomology in all even
%  degrees~\cite[Proposition 1.5]{Loeh}, so one may expect something
%  similar to hold for the ascending HNN-extension $F \wr \mathbb{Z}$.
%  On the other hand, Corollary~\ref{cor:HNN:nonsurj} shows that this is not
%  necessarily the case: A priori $F \wr \mathbb{Z}$ could even be
%  boundedly acyclic.
%\end{rem}

%%%%%%%

\section{Finitely generated groups with large bounded cohomology}
\label{sec:large}

In this section, we turn to groups with large bounded cohomology,
namely whose bounded cohomology is at least continuum-dimensional in
every degree at least~$2$
(Definition~\ref{defi:groups:with:large:bounded}). Countable
examples have been constructed before~\cite{Loeh}, but no
finitely generated example was known. Here, we provide the first
recipe for finitely generated examples:

\begin{thm}
  \label{thm:large}
  There exists a $6$-generated group with large bounded
  cohomology. Moreover, this group can be chosen to be torsion-free.
\end{thm}

The proof of Theorem~\ref{thm:large} will be completed in
Section~\ref{subsec:prooflarge}.

Combining this theorem with Theorem~\ref{thm:HNN}~\cite{HNN}, we
obtain:

\begin{cor}
  \label{cor:large:cont}
  There exist continuum many non-isomorphic $8$-generated groups
  with large bounded cohomology.
\end{cor}

\begin{proof}
  Let $\Gamma$ be a $6$-generated torsion-free group with large
  bounded cohomology (Theorem~\ref{thm:large}).  Then, for every group
  $\Lambda$ the product $\Lambda \times \Gamma$ also has large bounded
  cohomology, because the product retracts onto~$\Gamma$.  Now by
  Corollary~\ref{cor:HNN:cont} there exist continuum many
  $2$-generated groups pairwise distinguished by their torsion.
  Taking $\Lambda$ to be in this family, since $\Gamma$ is
  torsion-free, we obtain continuum many non-isomorphic $8$-generated
  groups with large bounded cohomology.
\end{proof}

\begin{rem}
Again, we can also choose these groups to be torsion-free using Remark~\ref{rem:continuum:tf:bac}.
\end{rem}

\subsection{Constructing groups with large bounded cohomology}

We provide a criterion to construct finitely
generated groups with large bounded cohomology, by starting with a
group that is isomorphic to a proper direct factor of itself. We
begin by introducing a local version of
Definition~\ref{defi:groups:with:large:bounded}:

\begin{defi}
  \label{defi:groups:with:large:bounded2}
  Let $n \geq 2$. A group $\Gamma$ has \emph{large $n$-th
    bounded cohomology} if $\dim_\R \HH^n_b(\Gamma; \R) \geq |\R|$.
\end{defi}

\begin{rem}
  Let $\Gamma$ be a group. 
  The following inequality always holds:
  $$\dim_\R \HH^n_b(\Gamma; \R) \leq |\CC^n_b(\Gamma; \R)| \leq |\R|^{\Gamma^{n+1}}$$
  In particular, if $\Gamma$ is countable, then $\dim_\R
  \HH^n_b(\Gamma; \R) \leq |\R|$. This shows that countable groups
  with large $n$-th bounded cohomology have $n$-th bounded cohomology
  of dimension equal to $|\R|$. For arbitrary groups, larger
  cardinalities are possible~\cite{Faiziev}.
\end{rem}

\begin{example}
  \label{ex:large}
  The following examples will be useful in the rest of this
  section:
  \begin{enumerate}
  \item Let $G *_C H$ be an amalgamated product with $|C \backslash G
    / C| \geq 3$ and $C \neq H$. Then $G *_C H$ has large second
    bounded cohomology~\cite{Grigorchuk:free, Fujiwara:free}.
  \item If $\Delta$ has large second bounded cohomology and $\Gamma$
    surjects onto $\Delta$, then $\Gamma$ has large second bounded
    cohomology (Section~\ref{subs:quot}).
  \item If $H$ is a retract of $\Gamma$,
  and $H$ has large $n$-th bounded cohomology,
  then so does $\Gamma$, since the epimorphism
  $\Gamma \to H$ induces an injection in bounded cohomology.
  A special case of this is when $H$ is a direct factor of $\Lambda$,
  which was used in the proof of Corollary~\ref{cor:large:cont}.
  \item Acylindrically hyperbolic groups have large
    second~\cite{Hull_Osin} and third~\cite{FPS} bounded cohomology.
    We will only need the case of
    fundamental groups of oriented closed connected hyperbolic $3$-manifolds
    for the proof of
    Theorem~\ref{thm:large}. This case was known earlier~\cite{Brooks, Epstein_Fujiwara,
      Soma}.
\end{enumerate}
A further example is contained in Lemma~\ref{lem:largeamalg}.
\end{example}

The main tool in the proof of Theorem~\ref{thm:large} is the
following:

\begin{thm}
  \label{thm:large2}
  Let $\Gamma$ and $\Sigma$ be groups such that $\Gamma \cong \Gamma \times
  \Sigma$. Suppose that $\Sigma$ has large second bounded cohomology. Then
  $\Gamma$ has large bounded cohomology in all even degrees.

  Moreover, if $\Lambda$ is the fundamental group of an oriented closed connected hyperbolic
  $3$-manifold, then
  $\Gamma \times \Lambda$ has large bounded cohomology.
\end{thm}

\begin{proof}
  For every $d \geq 1$ we can write $\Gamma$ as $\Gamma \times
  \Sigma^d$. Now it follows at once from Remark~\ref{rem:loneb} that
  $\Gamma$ has large bounded cohomology in all even degrees.  Next,
  since $\loneb 3(\Lambda) > 0$ (Example~\ref{exa:hyp}), again it
  follows from Remark~\ref{rem:loneb} that $\Gamma \times \Lambda$ has
  large bounded cohomology in all even degrees, and all degrees of the
  form $2d + 3$ for $d \geq 1$, which includes all integers $n \geq
  2$, except for $n = 3$.  Finally, $\Gamma
  \times \Lambda$ retracts onto $\Lambda$, and so has large third bounded cohomology by
  Examples~\ref{ex:large}.3 and ~\ref{ex:large}.4.
\end{proof}

\begin{rem}
\label{rem:large3}

  The last statement of Theorem~\ref{thm:large2} also holds if we only
  assume that $\HH^2_b(\Sigma; \R) \not \cong 0$.  Indeed, it is still true
  that $\Gamma \times \Lambda$ has large second and third
  bounded cohomology, and that $\loneb 2(\Gamma \times \Lambda) \geq |\R|$.  We can write all greater even integers as $2d + 2$ (Remark~\ref{rem:loneb}),
  thus obtaining largeness in all even degrees.  For the larger odd
  degrees we proceed as in the proof of Theorem~\ref{thm:large2}.
\end{rem}

\begin{rem}
  Notice that we might replace~$\Lambda$ in the product above
  by an arbitrary acylindrically hyperbolic group~$A$,
  if we knew
  that $\loneb 3(A) > 0$. This, however, seems to be an open problem,
  even for non-abelian free groups.
\end{rem}

The remaing part of this section is devoted to the construction of a
finitely generated group $\Gamma$ satisfying the assumptions of
Theorem~\ref{thm:large2}.

\subsection{Meier's finitely generated group}

Finitely generated groups $\Gamma$ with the property that $\Gamma
\cong \Gamma \times \Sigma$ for some $\Sigma \not\cong 1$ were first constructed by
Jones~\cite{Jones}.  Further constructions were given by
Meier~\cite{Meier}, Rhemtulla~\cite{Rhemtulla} and
Hirshon~\cite{Hirshon}. The last paper provides examples with $\Sigma$
finitely presented.  Among the constructions, particular attention has
been devoted to the case in which $\Gamma = \Sigma$.  We will show that
Meier's group satisfies the assumption of Theorem~\ref{thm:large2}.

\begin{defi}[Meier's finitely generated group]\label{defi:Meier}
  Let $$B \coloneqq \langle a, t \mid t^{-1}a^2t = a^3 \rangle$$ be
  the Baumslag--Solitar group $\BS(2, 3)$ and let $\overline{B}$ be
  another copy of $B$ with generators $\overline{a}, \overline{t}$.
  Let $$L \coloneqq \langle t, [a, t^{-1}at] \rangle \leq B,$$
  which is
  a free subgroup of rank $2$, and let $\overline{L}$ be its copy
  inside $\overline{B}$.  We then set $\Delta$ to be the free product of
  $B$ and $\overline{B}$ amalgamated over $L \cong \overline{L}$ by
  switching the generators; namely
  $$\Delta \coloneqq B \ast_{L \cong \overline{L}} \overline{B} = \langle B, \overline{B} \mid t = [\overline{a}, \overline{t}^{-1}\overline{a}\overline{t}], \overline{t} = [a, t^{-1}at] \rangle.$$
  Notice that $\Delta$ is torsion-free, being an amalgamated product of two
  torsion-free groups, and it is generated by $a, \overline{a}$ and
  $t$.  Let $\Gamma$ be the subgroup of $\Delta^{\mathbb{N}}$ generated by
  the diagonal elements $(a, a, \ldots), (\overline{a}, \overline{a}, \ldots),
  (t, t, \ldots)$ together with the element $(1, a, a^2, \ldots)$.
  Thus $\Gamma$ is a four-generated subgroup of $\Delta^{\mathbb{N}}$
  containing the diagonal.
  We call $\Gamma$ \emph{Meier's finitely generated group}.
\end{defi}

\begin{thm}[{\cite[Proposition 7]{Meier}}]
Meier's finitely generated group $\Gamma$ satisfies $\Gamma \cong \Gamma \times \Gamma$.
\end{thm}

\begin{rem}\label{rem:H2:T:implies:H2:Gamma}
  The projection $\Delta^{\N} \to \Delta$ onto the first coordinate restricts to
  an epimorphism $\psi \colon \Gamma \to \Delta$, because $\Gamma$ contains
  the diagonal of $\Delta^{\N}$.  Hence, by Example~\ref{ex:large}.2, if $\Delta$
  has large second bounded cohomology then so does~$\Gamma$.  Thus,
  in this case, $\Gamma$ satisfies the hypotheses of Theorem~\ref{thm:large2}.
\end{rem}

The previous remark shows that we reduced the problem of computing
the second bounded cohomology of Meier's finitely generated group~$\Gamma$ to
the one of computing $\HH^2_b(\Delta; \R)$.  Recall from
Example~\ref{ex:large}.1 that an amalgamented free product $G *_C H$ has
large second bounded cohomology provided $|C \backslash G / C| \geq 3$
and $C \neq H$. We will spend the last part of this section by proving
that $T$ satisfies the previous condition. 

Before proving that $|L \backslash B / L|$ has the desired
cardinality, it is useful to recall the proof that $B$ is
non-Hopfian~\cite{BS}: This amounts in constructing a non-injective
self-epimorphism $\varphi$.  The desired $\varphi \colon B \to B$ is
defined on the standard generators of $B$ by $\varphi(a) = a^2,
\varphi(t) = t$. Since the image of $\varphi$ contains both the
generators $a = \varphi([t, a^{-1}])$ and $t = \varphi(t)$, the
homomorphism~$\varphi$
is surjective.  On the other hand, $\varphi$ is non-injective because
$\varphi([a, t^{-1}at]) = 1$, and this element is easily seen to be
non-trivial using Britton's Lemma~\cite[Theorem 11.81]{rotman}.

We are now ready to show that the cardinality $|L \backslash B / L|$
is infinite, by using the previous homomorphism $\varphi$.

\begin{lemma}
  \label{lem:dc}
  Let $\varphi \colon B \to B$ be the non-injective self-epimorphism
  defined above.  Then, there exists a strictly nested sequence $$L <
  \varphi^{-1}(L) < \varphi^{-2}(L) < \cdots .$$ In particular $|L
  \backslash B / L| = \infty$.
\end{lemma}

\begin{proof}
  Since $[a, t^{-1}at] \in \ker(\varphi)$, we know that
  $$\varphi(L) = \langle \varphi(t), \varphi([a, t^{-1}at]) \rangle = \langle t \rangle.$$
  In particular $L$ is strictly contained in the group
  $\varphi^{-1}(L)$. Applying $\varphi^{-1}$ repeatedly to both sides
  we obtain the desired sequence.

  Now let $x \in \varphi^{-1}(L) \, \backslash \, L$. Then the double
  cosets $L$ and $LxL$ are distinct and contained in
  $\varphi^{-1}(L)$. We repeat this argument inductively on the
  sequence, and obtain that $\varphi^{-k}(L)$ contains at least
  $(k+1)$ distinct double cosets of $L$. Thus $|L \backslash B / L| =
  \infty$.
\end{proof}

The previous discussion leads to the following:

\begin{prop}\label{prop:Meier:2nd:bc}
  Meier's finitely generated group~$\Gamma$ (Definition~\ref{defi:Meier})
  has large second bound\-ed cohomology.
\end{prop}
\begin{proof}
  By Remark~\ref{rem:H2:T:implies:H2:Gamma}, we know that if $\Delta$ has
  large second bounded cohomology, then so does $\Gamma$.  On the
  other hand, Lemma~\ref{lem:dc} and Example~\ref{ex:large}.1
  immediately imply that $\Delta$ has large second bounded cohomology.
\end{proof}

\subsection{Proof of Theorem~\ref{thm:large}}\label{subsec:prooflarge}

We now have all the tools for giving a short proof of
Theorem~\ref{thm:large}:

\begin{proof}[Proof of Theorem~\ref{thm:large}]
  Let $\Gamma$ be Meier's finitely generated group, which is
  four-generated and torsion-free. We proved in
  Proposition~\ref{prop:Meier:2nd:bc} that $\Gamma$ has large second
  bounded cohomology and so it satisfies the hypothesis of
  Theorem~\ref{thm:large2}.
  
  Let $\Lambda$ be the fundamental group of an oriented closed
  connected hyperbolic $3$-manifold. We can choose $\Lambda$ to be
  $2$-generated by proceeding as follows.  Start with the figure-eight
  knot complement $M$, whose fundamental group is
  $2$-generated~\cite[Example VI.4.3]{knots}.  Then perform a suitable Dehn
  Filling on $M$ to obtain an oriented closed connected hyperbolic
  $3$-man\-i\-fold~$N$~\cite[Chapter 4]{Thurston}. The fundamental group
  $\Lambda$ of $N$ is obtained from the one of $M$ by adding relations.
  In particular it is still $2$-generated.
  
  The thesis
  now follows by taking $\Gamma \times \Lambda$, which is a
  torsion-free $6$-generated group, and applying
  Theorem~\ref{thm:large2}.
\end{proof}

\begin{scholium}[On finitely presented examples]
\label{scholium:fp}
  The criterion in Theorem~\ref{thm:large2}
  does not seem to help to construct finitely presented groups
  with large bounded cohomology. Indeed, as remarked by
  Hirshon in~1994: ``The question whether or not there exists a
  finitely presented group which is isomorphic to a proper direct
  factor of itself is an unsolved problem which has been around for a
  long time''~\cite{prod_fp2}.

  The most promising constructions in this direction are given by
  finitely presented groups that allow epimorphisms onto their direct
  square. These were first constructed by Baumslag and
  Miller~\cite{odd_fp}. Further examples were given by Hirshon and
  Meier~\cite{prod_fp}, see also~\cite{prod_fp2} for a list of
  striking properties of these groups.
  
  The group~$H$ constructed by Baumslag and Miller~\cite{odd_fp}
  surjects onto the group~$\Delta$ appearing in
  Definition~\ref{defi:Meier}.  Since $\Delta$ has large second bounded
  cohomology (Lemma~\ref{lem:dc} and Example~\ref{ex:large}.1), we
  then have $\dim_\R \HH^2_b(H; \R) \geq |\R|$. However, the
  description of the full bounded cohomology of~$H$ is strongly
  related to a better understanding of $\lex$ groups
  (Section~\ref{subs:quot}).  Indeed, since $H$ surjects onto~$H^d$,
  we can produce quotients of~$H$ with large bounded cohomology in
  every given even degree, but we cannot obtain any information about
  the bounded cohomology of~$H$ itself, if we do not know whether
  direct powers of~$H$ are~$\lex$.

  In Section~\ref{sec:largefp}, we give an ad-hoc example of
  a finitely presented group with large bounded cohomology.
\end{scholium}

\begin{scholium}[On the bounded cohomology of the free group]
The connection between finitely generated groups with large bounded
cohomology and $\lex$ groups goes even further.  Indeed, if $\Gamma$
is an $n$-generated group with large bounded cohomology, and $\Gamma$
is moreover $\lex$, then the free group $F_n$ has large bounded
cohomology.  In particular, if the group from Theorem~\ref{thm:large}
is $\lex$, then $F_6$ has large bounded cohomology.  As mentioned in
Example~\ref{ex:large}.4, it is known that non-abelian free groups
have large second and third bounded cohomology.  However, 
it is not known whether the bounded cohomology of 
non-abelian free groups vanish or not in degrees 
$4$ and above~\cite[Question~18.3]{BIMW}.

Note that proving that $F_6$ has large bounded cohomology
would have many strong consequences.
 Indeed, every acylindrically hyperbolic group
admits a hyperbolically embedded subgroup of the form $F_n \times K$,
where $K$ is finite, for all $n \geq 1$~\cite[Theorems 6.8 and
  6.14]{DGO}.  It then follows that if some non-abelian free group had
large bounded cohomology, then the same would be true of all
acylindrically hyperbolic groups~\cite[Theorem 1.1]{FPS} .
\end{scholium}

%%%%%%%%%%%%%%%%%%%%%%%%%%%%%%%%%%%%%%%%%%%%%%%%%%%%

\section{Finitely presented groups with large bounded cohomology}
\label{sec:largefp}

In Section~\ref{sec:large}, we gave a general recipe for producing finitely generated groups with large bounded cohomology.
As remarked in Scholium~\ref{scholium:fp}, this does not allow to construct finitely presented examples, given the current limited list of examples of groups isomorphic to their own direct factors.

In this section, we produce finitely presented (in fact, type~$F_\infty$) groups with large bounded cohomology, by using an ad-hoc construction, and building on previous work.

The main player is \emph{Thompson's group~$T$}:

\begin{defi}
  The Thompson group~$T$ is the group of orientation-pre\-ser\-ving
  piecewise linear homeomorphisms~$f$
  of the circle~$\R/\Z$ with the following properties:
\begin{enumerate}
\item $f$ has finitely many breakpoints, all of which lie in~$\Z[1/2]/\Z$;
\item Away from the breakpoints, the slope of~$f$ is a power of~$2$;
\item $f$ preserves~$\Z[1/2]/\Z$.
\end{enumerate}
\end{defi}

The group~$T$ and its siblings $F$ and~$V$ were introduced by Richard Thompson
in~1965; they are some of the most important groups in geometric and
dynamical group theory. We refer the reader to the literature~\cite{thompson} for a
detailed discussion.

The group $T$ is finitely presented, and even of
type~$F_\infty$.
The integral cohomology of~$T$ has been computed~\cite{cohoT}, and this result has interesting consequences for its bounded cohomology, as was first noted by Burger and Monod:

\begin{prop}[Ghys--Sergiescu \cite{cohoT}, Burger--Monod \cite{rigidity}]
\label{prop:cupT}

For each even integer~$n \geq 2$, we have~$\HH^n_b(T; \mathbb{R}) \not \cong 0$.
\end{prop}

This result is obtained by noticing that every cup power of the Euler class of~$T$ is bounded, and using the fact that these cup powers are non-zero in ordinary cohomology~\cite{cohoT}.

\begin{thm}\label{thm:largefp}
Let $T$ be Thompson's group, and let $\Lambda$ be the fundamental group of an oriented closed connected hyperbolic $3$-manifold. Then $T \times \Lambda$ has large bounded cohomology.

Moreover, $T \times \Lambda$ is finitely presented, in fact type $F_\infty$.
\end{thm}

\begin{proof}
By Proposition~\ref{prop:cupT}, we have $\HH^n_b(T; \mathbb{R}) \not \cong 0$ for every even integer~$n \geq 2$.
Using that $\Lambda$ has large second and third bounded cohomology, and $\loneb 3(\Lambda) > 0$, we conclude in the same way as in the proof of Theorem~\ref{thm:large2} (see also Remark~\ref{rem:large3})
\end{proof}

%%%%%%%%%%%%%%%%%%%%%%%%%%%%%%%%%%%%%%%%%%%%%%%%%%%%%%%%%%%
\section{Non-computability}\label{sec:noncomp}

In the following, we prove Theorem~\ref{ithm:noncomp} and
Theorem~\ref{ithm:noncompspaces}, i.e., we show that many decision
problems concerning bounded cohomology are \emph{not} algorithmically
decidable:

\begin{thm}[Non-computability of bounded cohomology]\label{thm:noncomp}
  Let $d \in \N_{\geq 2}$ and let $D \in \N$.  Then all of the
  following algorithmic problems are undecidable: Given a finite
  presentation~$\genrel SR$, decide whether
  \begin{enumerate}
  \item \label{nc:trivial}
    $\HH^{d}_b(\genrel SR;\R) \cong 0$ or not; 
  \item \label{nc:dim}
    $\dim_\R \HH_b^{d}(\genrel SR;\R) \leq D$ or not;
  \item \label{nc:infdim}
    $\HH^{d}_b(\genrel SR;\R)$ is large (Definition~\ref{defi:groups:with:large:bounded2}) or not;
  \item \label{nc:l1}
    $\lonehom {d} (\genrel SR;\R) \cong 0$ or not;
  \item \label{nc:l1red}
    $\loneb {d} (\genrel SR) > 0$ or not; 
  \item \label{nc:bacyclic}
    $\genrel SR$ is boundedly acyclic or not;
  \item \label{nc:cdb}
    $\cdb \genrel SR \leq D$ or not. 
  \item \label{nc:hdb}
    $\hdb \genrel SR \leq D$ or not.
  \end{enumerate}
\end{thm}

Basic terminology and properties of bounded (co)homological dimension
are recorded in Appendix~\ref{appx:cdb}.

\begin{cor}[Non-computability of bounded cohomology for spaces]\label{cor:noncompspaces}
  Let $d \in \N_{\geq 2}$, and let $D \in \N$. Then all of the
  following algorithmic problems are undecidable: Given a finite
  simplicial complex~$X$, decide whether
  \begin{enumerate}
  \item\label{ncs:trivial} $\HH_b^{d}(X;\R) \cong 0$ or not;
  \item\label{ncs:dim} $\dim_\R \HH_b^{d}(X;\R) \leq D$ or not;
  \item\label{ncs:infdim} $\HH_b^{d}(X;\R)$ is large or not;
  \item $\lonehom {d}(X;\R) \cong 0$ or not;
  \item $X$ is boundedly acyclic or not;
  \end{enumerate}
\end{cor}

Here, finite simplicial complexes~$X$ are given as the finite set~$V$
of their vertices together with the finite subsets of~$V$ that
constitute simplices of~$X$. For the definition of bounded cohomology
and $\ell^1$-homology of spaces we refer to the
literature~\cite{vbc}\cite[Chapters~5 and 6]{Frigerio}.

%%%%%%%%%%%%%%%%%%%%%%%%
\subsection{Proof of Theorem~\ref{thm:noncomp}, parts~(\ref{nc:cdb}), (\ref{nc:hdb})}

We use the Adian--Rabin Theorem. For the sake of completeness, we
recall the basics; we refer to Rotman's book~\cite[Chapter~12]{rotman}
for further details and history.

\begin{defi}[Markov property]
  A subclass~$P$ of the class of all finitely presentable groups is a
  \emph{Markov property} if the following properties all hold:
  \begin{itemize}
  \item The class~$P$ is closed under taking isomorphisms.
  \item \emph{Positive witness.} The class~$P$ is non-empty.
  \item \emph{Negative witness.} There exists a finitely presentable
    group that is \emph{not} isomorphic to a subgroup of an element
    of~$P$.
  \end{itemize}
\end{defi}

\begin{thm}[Adian--Rabin~\protect{\cite[Theorem~12.32]{rotman}}]\label{thm:adianrabin}
  Let $P$ be a Markov class of finitely presentable groups. Then the
  following algorithmic problem is undecidable:
  \begin{itemize}
  \item[] Given a finite presentation~$\genrel SR$, decide whether
    $\genrel SR$ lies in~$P$ or not.
  \end{itemize}
\end{thm}
  
\begin{lemma}\label{lem:cdbmarkov}
  Let $D \in \N$.
  \begin{enumerate}
  \item
    The class of all finitely presentable groups~$\Gamma$ with~$\cdb
    \Gamma \leq D$ is a Markov property.
  \item
    The class of all finitely presentable groups~$\Gamma$ with~$\hdb
    \Gamma \leq D$ is a Markov property.
  \end{enumerate}
\end{lemma}
\begin{proof}
  Clearly, these classes of groups are closed under isomorphisms.

  \emph{Positive witness.} The trivial group is a positive witness.

  \emph{Negative witness.}  In view of the hyperbolic examples in
  Example~\ref{exa:hdbcdb}, there exists a finitely presentable
  group~$\Lambda$ with~$\hdb \Lambda > D$ and $\cdb \Lambda > D$.  By
  the monotonicity of bounded (co)homological dimension
  (Proposition~\ref{prop:dimmono}), it follows that $\Lambda$ is
  \emph{not} isomorphic to subgroups of groups~$\Gamma$ with~$\hdb
  \Gamma \leq D$ or~$\cdb \Gamma\leq D$. Hence, these groups provide 
  the desired negative witnesses.
\end{proof}

\begin{proof}[Proof of Theorem~\ref{thm:noncomp}.(\ref{nc:cdb}), (\ref{nc:hdb})]
  The claims follow by applying the Adian--Rabin Theorem
  (Theorem~\ref{thm:adianrabin}) to the Markov class of finitely
  presentable groups~$\Gamma$ with~$\cdb \Gamma \leq D$ or $\hdb
  \Gamma \leq D$, respectively (Lemma~\ref{lem:cdbmarkov}).
\end{proof}

\begin{rem}[(non-)Markov properties]\label{rem:markov}
  \hfil
  \begin{itemize}
  \item 
    Amenability of finitely presentable groups is a Markov property:
    The trivial group is a positive witness; the free group of
    rank~$2$ is a negative witness.
  \item
    However, bounded acyclicity of finitely presentable groups is
    \emph{not} a Markov property: In view of Corollary~\ref{cor:fp},
    there do \emph{not} exist negative witnesses.
  \item
    Not being boundedly acyclic is \emph{not} a Markov property of
    finitely presentable groups. Assume by contradiction that there
    exists a negative witness~$\Gamma$. Then $\Gamma$ is isomorphic to
    a subgroup of~$\Gamma \times F_2$ and $\Gamma \times F_2$ is not
    boundedly acyclic by Proposition~\ref{prop:loneb}.
  \item
    Similarly, (not) having large bounded cohomology is not a Markov property.
    One can obtain this via the same reasoning as before,
    using the group from Theorem~\ref{thm:fp:large:bdd} instead of the free group.
  \end{itemize}
\end{rem}

In view of Remark~\ref{rem:markov}, the Adian--Rabin Theorem
(Theorem~\ref{thm:adianrabin}) cannot be applied directly to establish
the remaining parts of Theorem~\ref{thm:noncomp}.

%%%%%%%%%%%%%%%%%%%%%%%%
\subsection{Proof of Theorem~\ref{thm:noncomp}, parts~(\ref{nc:trivial})--(\ref{nc:bacyclic})}

We use the standard technique of turning the word problem into group
presentations.

\begin{construction}\label{constr:rabinlemma}
  By the Novikov--Boone--Britton Theorem~\cite[Theorem~12.8]{rotman},
  there exists a finitely presented group~$\Lambda$ with unsolvable
  word problem; let $\genrel SR$ be a finite presentation of~$\Lambda$
  with symmetric generating set~$S$.  By the proof of the Adian--Rabin
  Theorem~\cite[Lemma~12.31]{rotman}, there exists an algorithm
  \begin{align*}
    \text{words over~$S$}
    & \longrightarrow
    \text{finite presentations}
    \\
    w
    & \longmapsto \genrel {S_w}{R_w}
  \end{align*}
  with the following property: For all words~$w$ over~$S$, we have
  \[ \text{$w$ represents the neutral element of~$\Lambda$}
  \Longleftrightarrow
  \genrel {S_w}{R_w} \cong 1.
  \]
  For a word~$w$ over~$S$, we write~$\Lambda_w := \genrel{S_w}{R_w}$.

  In addition, this construction can be refined as follows: We can take
  $\Lambda$ to be torsion-free~\cite[Theorem~12 on p.~88]{millerbook}
  and we can assume that there is an algorithm
  \begin{align*}
    \text{words over~$S$} & \longrightarrow \text{words over~$S_w$}
    \\
    w & \longmapsto \overline w
  \end{align*}
  with the following property~\cite[(proof of) Lemma~12.31]{rotman}:
  If $w$ does not represent the neutral element in~$\Lambda$, then
  $\overline w$ has infinite order in~$\Lambda_w$.
\end{construction}

Via the groups~$\Lambda_w$, we can reduce the algorithmic problems in
Theorem~\ref{thm:noncomp}.(\ref{nc:trivial})---(\ref{nc:bacyclic}) to
the word problem in~$\Lambda$. As a preliminary stage, we take the
free product with~$\Z$ (this will be sufficient for the degrees $2$
and~$3$, as well as for the claim on bounded acyclicity):

\begin{construction}\label{constr:rabinZ}
  In the situation of Construction~\ref{constr:rabinlemma}, for
  words~$w$ over~$S$, we consider
  \[ \Gamma_w := \genrel{S_w,t}{R_w} \cong \Lambda_w * \Z,
  \]
  where $t \not\in S$ is a new generator. By construction, we have:
  \begin{itemize}
  \item If $w$ represents the neutral element of~$\Lambda$,
    then
    \[ \Gamma_w \cong \Lambda_w * \Z \cong 1 * \Z \cong \Z.\]
    In particular, $\Gamma_w$ is amenable.
  \item If $w$ does \emph{not} represent the neutral element
    of~$\Lambda$, then $\Lambda_w \not\cong 1$ and so $\Gamma_w$ is a
    non-elementary free product, i.e., a free product~$A * B$, where
    $A$ and $B$ are non-trivial and we do not have~$A \cong \Z/2 \cong
    B$.  In particular, $\Gamma_w$ is an acylindrically hyperbolic
    group~\cite[Corollary~2.2]{Minasyan:Osin}.
  \end{itemize}
  Therefore, in bounded cohomology, we obtain:
  \begin{itemize}
  \item If $w$ represents the neutral element of~$\Lambda$, then
    $\Gamma_w$ is amenable and so it is boundedly acyclic; in particular,
    $\Gamma_w$ also has trivial reduced and unreduced
    $\ell^1$-homology~\cite[Corollary~2.4]{MM}.
  \item
    If $w$ does \emph{not} represent the neutral element of~$\Lambda$,
    then $\HH^2_b(\Gamma_w;\R)$ and $\HH^3_b(\Gamma_w;\R)$ are large
    (Example~\ref{ex:large}). 
    Moreover, $\loneb 2(\Gamma) = |\R|$ (Proposition~\ref{prop:loneb}).
  \end{itemize}
\end{construction}

For the higher degree case of Theorem~\ref{thm:noncomp}, we use
the witness construction of Weinberger~\cite[Chapter~2.6]{weinberger}.

\begin{construction}\label{constr:rabinamalg}
  In the situation of Construction~\ref{constr:rabinlemma}, we consider
  the following witness groups: Let $\Gamma = \genrel {S'}{R'}$ be a
  torsion-free finitely presented group, where $S' = \{s'_1, \dots, s'_k\}$
  and where all elements of~$S'$ are non-trivial in~$\Gamma$.
  For words~$w$ over~$S$, we define
  \[ W(\Gamma,\Lambda,w)
     \coloneqq \Gamma *_\Z \Lambda_w *_\Z \Lambda_w *_\Z \dots *_\Z \Lambda_w,
  \]
  where the $k$-fold push-out group is given by the glueings
  of~$\Gamma$ and~$\Lambda_w$ over the maps~$s'_j \mapsfrom 1 \mapsto
  \overline w$ for~$j \in \{1,\dots, k\}$.
  In other words:
  $$W(\Gamma, \Lambda, w) = \bigl\langle \Gamma; (\Lambda_w)_j : j = 1, \ldots, k \bigm| s'_j = \overline{w} \in (\Lambda_w)_j : j = 1, \ldots, k \bigr\rangle.$$
  The whole construction is
  algorithmic in the sense that we can also algorithmically give 
  finite presentations of~$W(\Gamma,\Lambda,w)$.
  By construction, we have:
  \begin{itemize}
  \item If $w$ represents the neutral element of~$\Lambda$,
    then~$W(\Gamma,\Lambda,w) \cong 1$, because all generators of~$\Gamma$
    are killed and $\Lambda_w \cong 1$.
  \item If $w$ does \emph{not} represent the neutral element of~$\Lambda$,
    then the construction of~$W(\Gamma,\Lambda,w)$ is a proper iterated
    $k$-fold amalgamated free product, because $s'_j$ and $\overline w$
    have infinite order in~$\Gamma$ and $\Lambda_w$, respectively. Moreover,
    the amalgamation is performed over the amenable group~$\Z$. 
    In particular, for~$d \in \N_{\geq 2}$, we have
    (Lemma~\ref{lem:largeamalg}): If $\HH_b^d(\Gamma;\R)$ is large,
    then also $\HH_b^d(W(\Gamma,\Lambda,w);\R)$ is large. Finally,
    $\loneb d (W(\Gamma,\Lambda,w)) \geq \loneb d (\Gamma)$.
  \end{itemize}
\end{construction}

\begin{lemma}\label{lem:largeamalg}
  Let $\Gamma$ and $\Lambda$ be countable groups, let $\Gamma *_A
  \Lambda$ be an amalgamated free product over a common amenable subgroup~$A$,
  and let $d \in \N_{>0}$.
  \begin{enumerate}
  \item If $\HH_b^d(\Gamma;\R)$ is large, then so
    is~$\HH_b^d(\Gamma*_A\Lambda;\R)$.
  \item We have~$\loneb d (\Gamma *_A \Lambda) \geq \loneb d (\Gamma)$.
  \end{enumerate}
\end{lemma}
\begin{proof}
  Let $i \colon \Gamma \hookrightarrow \Gamma *_A \Lambda$
  denote the canonical inclusion. Then, by~\cite[Theorem~1]{isometric}, there exists an isometric
  embedding~$\Theta \colon \HH_b^d(\Gamma;\R) \longrightarrow \HH_b^d(\Gamma *_A \Lambda;\R)$
  with
  \[ \HH_b^d(i) \circ \Theta = \id_{\HH_b^d(\Gamma;\R)}.
  \]
  This immediately gives the first part.

  For the second part, it suffices to show that the map~$\rlonehom d (i)
  \colon \rlonehom d (\Gamma;\R) \longrightarrow \rlonehom d(\Gamma
  *_A \Lambda ;R)$ is injective: Let $\alpha \in \rlonehom
  d(\Gamma;\R)$ with $\alpha \neq 0$. Then, by duality, there exists
  a~$\varphi \in \HH_b^d(\Gamma;\R)$ with~$\langle \varphi
  ,\alpha\rangle \neq 0$~\cite{MM}. Therefore, we obtain
  \begin{align*}
    \bigl\langle \Theta(\varphi)
    , \rlonehom d(i) (\alpha)
    \bigr\rangle
    = \bigl\langle \HH_b^d(i) \circ \Theta(\varphi)
    , \alpha
    \bigr\rangle
    = \langle \varphi,\alpha\rangle
    \neq 0.
  \end{align*}
  In particular, $\rlonehom d (i)(\alpha) \neq 0$ in~$\rlonehom d(\Gamma *_A \Lambda;\R)$. 
\end{proof}

\begin{proof}[Proof of Theorem~\ref{thm:noncomp}.(\ref{nc:trivial})---(\ref{nc:bacyclic})]
  We consider~$w \longmapsto \Gamma_w$ as in
  Construction~\ref{constr:rabinZ}.

  In degrees~$2$ and~$3$ and for the case of bounded acyclicity, we
  have: If any of the problems~(\ref{nc:trivial})--(\ref{nc:bacyclic})
  were decidable, then the computations in
  Construction~\ref{constr:rabinZ} show that also the word problem
  for~$\Lambda$ would be decidable.  This contradicts the choice
  of~$\Lambda$.
  
  For the higher degrees, we proceed as follows: Let $d \geq 4$.
  Let $\Gamma$ be the fundamental group of an oriented closed connected
  hyperbolic $(d-2)$-manifold,
  which is finitely presented and torsion-free.
  We then consider the construction
  \begin{align*}
    \text{words over~$S$} & \longrightarrow \text{finite presentations}
    \\
    w & \longmapsto \Pi_w \coloneqq W(F_2,\Lambda,w) \times W(\Gamma,\Lambda,w)
  \end{align*}
  from Construction~\ref{constr:rabinamalg} (with the product presentation). 
  Then, we obtain from the computations in Construction~\ref{constr:rabinamalg}:
  \begin{itemize}
  \item
    If $w$ represents the neutral element of~$\Lambda$, then $\Pi_w
    \cong 1$ is boundedly acyclic and has trivial reduced and
    unreduced $\ell^1$-homology.
  \item
    If $w$ does \emph{not} represent the neutral element of~$\Lambda$,
    then Remark~\ref{rem:loneb}, Example~\ref{exa:hyp},
    Example~\ref{ex:large}.4 and Lemma~\ref{lem:largeamalg} show that
    \begin{align*}
       \loneb d(\Pi_w)
       & \geq \loneb 2 \bigl( W(F_2,\Lambda,w) \bigr) \cdot \loneb {d-2} \bigl(W(\Gamma,\Lambda,w)\bigr)
       \\
       & \geq \loneb 2 (F_2) \cdot \loneb {d-2} (\Gamma)
       \geq |\R| \cdot 1. 
    \end{align*}
    In particular, $\HH_b^d(\Pi_w;\R)$ is large (Proposition~\ref{prop:loneb}).
  \end{itemize}
  Again, we see: If any of the
  problems~(\ref{nc:trivial})--(\ref{nc:l1red}) were decidable in
  degree~$d$, then the word problem for~$\Lambda$ would be decidable,
  which contradicts the choice of~$\Lambda$.
\end{proof}

\begin{rem}
\label{rem:compl2}

  The same argument as in the proof of the
  parts~(\ref{nc:trivial})--(\ref{nc:bacyclic}) of
  Theorem~\ref{thm:noncomp} also gives simple proofs of the following:
  \begin{enumerate}
  \item Let $d\in \N$. Then the $d$-th $L^2$-Betti number~$\ltb d$ of
    finitely presentable groups is \emph{not} computable.

    It should be noted that more sophisticated non-computability
    results for $L^2$-Betti numbers are already
    known~\cite{grabowski}.
  \item The cost of finitely presentable groups is \emph{not}
    computable.
  \end{enumerate}
  We refer the reader to~\cite{lueckl2} for the definition of $L^2$-Betti numbers, and to~\cite{gaboriau} for the definition of cost.  
  For the proofs, we use the same notation as above.
  
  \emph{Ad~1.}  
  If $d = 0$, then computing~$\ltb d$ amounts to computing the
  cardinality of the group in question~\cite[Theorem~6.54]{lueckl2};
  however, it is known that the cardinality of groups is \emph{not}
  computable from finite presentations~\cite[Corollary~12.33]{rotman}.

  We now let $d \geq 1$. 
  For words~$w$ over~$S$, we consider
  \[ \Delta_w := \Gamma_w \times (F_2)^{\times (d-1)}
  \]
  (which admits a finite presentation that can be algorithmically
  constructed from~$\genrel {S_w}{R_w}$). 
  The standard inheritance properties of $L^2$-Betti
  numbers~\cite[Theorem~6.54]{lueckl2} show:
  \begin{itemize}
  \item If $w$ represents the neutral element of~$\Lambda$,
    then $\Delta_w$ is isomorphic to~$\Z \times (F_2)^{\times(d-1)}$
    and thus~$\ltb d(\Delta_w) = 0$.
  \item
    If $w$ does \emph{not} represent the neutral element of~$\Lambda$,
    then
    \begin{align*}
      \ltb 1 (\Gamma_w)
    & = \ltb 1 (\Lambda_w * \Z)
      = \ltb 1 (\Lambda_w) + 1 - \ltb 0 (\Lambda_w)
      \\
    & > \ltb 1 (\Lambda_w) + 1 - 1
      \geq 0
    \end{align*}
    and so
    \begin{align*}
      \ltb d (\Delta_w)
      & = \ltb d (\Gamma_w \times (F_2)^{\times (d-1)})\\
      & \geq \ltb 1 (\Gamma_w) \cdot \ltb {d-1} \bigl((F_2)^{\times (d-1)}\bigr)
     = \ltb 1 (\Gamma_w) \cdot 1
      \\
      & > 0.
    \end{align*}
  \end{itemize}
  Therefore computability of $L^2$-Betti numbers would imply
  solvability of the word problem in~$\Lambda$, which contradicts
  the choice of~$\Lambda$.
    
  \emph{Ad~2.}
  For words~$w$ over~$S$, we again consider the group~$\Gamma_w$ as above.
  \begin{itemize}
  \item If $w$ represents the neutral element of~$\Lambda$, then
    $\Gamma_w \cong \Z$;
    in particular,~$\cost \Gamma_w = 1$~\cite[Corollaire~III.4]{gaboriau}.
  \item If $w$ does \emph{not} represent the neutral element
    of~$\Lambda$, then $\Lambda_w \not\cong 1$. Because finitely
    presented groups are countable and have finite cost, we
    obtain~\cite[Th\'eor\`eme~VI.7$''$]{gaboriau}
    \begin{align*}
      \cost (\Gamma_w)
      & = \cost(\Lambda_w * \Z)
      \geq \cost (\Lambda_w) + \cost(\Z)
      %\\
      %&
      > 0 + 1.
    \end{align*}
  \end{itemize}
  Therefore computability of cost would imply solvability of the word
  problem in~$\Lambda$, which contradicts the choice of~$\Lambda$.
\end{rem}

%%%%%%%%%%%
\subsection{Proof of Corollary~\ref{cor:noncompspaces}}

We deduce Corollary~\ref{cor:noncompspaces} from Theorem~\ref{thm:noncomp}
via Gromov's Mapping Theorem~\cite{vbc, ivanov, FM:Grom}:

\begin{proof}[Proof of Corollary~\ref{cor:noncompspaces}]
  We proceed by contradiction: In view of Theorem~\ref{thm:noncomp} it
  suffices to show that if any of the decision problems formulated in
  Corollary~\ref{cor:noncompspaces} were decidable, then the
  corresponding decision problem for groups in
  Theorem~\ref{thm:noncomp} would be decidable.

  There is an algorithm
  \[ P \colon \text{finite presentations}
     \longrightarrow \text{finite simplicial complexes}
  \]
  with the following property: For all finite presentations~$\genrel
  SR$, the simplicial complex~$P(\genrel SR)$ is connected and
  \[ \pi_1\bigl( P(\genrel SR) \bigr) \cong \genrel SR.
  \]
  For example, one such construction is to take the double barycentric
  subdivision of the presentation cellular complex of the
  presentation~$\genrel SR$.

  By Gromov's Mapping Theorem~\cite{vbc, ivanov, FM:Grom}, we know that
  \[ \fa{d \in \N} \HH_b^d\bigl( P(\genrel SR);\R)
     \cong \HH_b^d\bigl(\genrel SR;\R\bigr);
  \]
  moreover, we have the analogous statement for
  $\ell^1$-homology~\cite[Corollary~5.2]{loehl1}:
  \[ \fa{d \in \N} \lonehom d\bigl(P(\genrel SR);\R\bigr)
     \cong \lonehom d \bigl(\genrel SR;\R\bigr).
  \] 
  Therefore, any algorithm for the problems formulated in
  Corollary~\ref{cor:noncompspaces} would lead to a corresponding
  algorithm for the bounded cohomology of finitely presented
  groups. This contradicts Theorem~\ref{thm:noncomp} and thus
  completes the proof of Corollary~\ref{cor:noncompspaces}.
\end{proof}

%%%%%%%%%%%%%%%%%%%%%%%%%%%%%%%%%%%%%%%%%%%%%%%%%%%%%%%%%%%%%
\appendix
\section{Bounded (co)homological dimension}\label{appx:cdb}

We introduce dimension notions of groups in the context of bounded
cohomology and $\ell^1$-homology following previous works by
Johnson~\cite{Johnson2, Johnson} and
Monod~\cite[Definition~3.1]{Monod:inv}. In contrast with the notion~\cite{Loeh}
\[ \bcd \Gamma := \sup \bigl\{ n \in \N \bigm| \HH^n_b(\Gamma;\R) \not\cong 0 \bigr\}
\in \N \cup \{\infty\}
\]
these dimensions mimic the usual (co)homological dimension and thus
take twisted coefficients into account.

\begin{defi}[bounded (co)homological dimension]
  Let $\Gamma$ be a group.
  \begin{itemize}
  \item
    The \emph{bounded cohomological dimension} of~$\Gamma$
    is defined as
    \[ \cdb \Gamma := \sup \bigl\{ n \in \N
                           \bigm| \exists_{V \in \Ban_\Gamma}\; \HH^n_b(\Gamma;V) \not\cong 0
                           \bigr\}
                   \in \N \cup \{\infty\}.        
    \]
  \item The \emph{bounded homological dimension} of~$\Gamma$
    is defined as
    \[\hdb \Gamma := \sup \bigl\{ n \in \N
                           \bigm| \exists_{V \in \Ban_\Gamma}\; \lonehom n (\Gamma;V) \not\cong 0
                           \bigr\}
                   \in \N \cup \{\infty\}.        
    \]
  \end{itemize}
\end{defi}

\begin{rem}\label{rem:hdbcdb}
  Let $\Gamma$ be a group. By
  duality~\cite[Corollary~2.4.2]{MM}\cite[Corollary~5.3]{loehl1},
  we have
  \[ \hdb \Gamma \leq \cdb \Gamma.
  \]
  More precisely: Let $n \in \N$ and let $V$ be a Banach
  $\Gamma$-module with~$\HH^k_b(\Gamma;V') \cong 0$ for all~$k \in
  \N_{>n}$.  Then $\lonehom k (\Gamma;V) \cong 0$ for all~$k \in
  \N_{>n}$~\cite[Corollary~5.3]{loehl1}\cite[Remark~3.7]{loehthesis}.
\end{rem}

\begin{example}\label{exa:hdbcdb}
  \hfil
  \begin{itemize}
  \item Let $\Gamma$ be a group. Then $\cdb \Gamma = 0$ if and only
    if~$\Gamma$ is finite~\cite[Theorem~3.12]{Frigerio}. 
  \item Let $\Gamma$ be a group. Then $\hdb \Gamma = 0$ if and only
    if~$\Gamma$ is amenable~\cite[Corollary~5.5]{loehl1}.
  \item
    In particular,  $\hdb \Z = 0 < \cdb \Z$.
  \item If $\Gamma$ is a non-amenable group, then $\cdb \Gamma \geq
    3$.  This can be derived from the fact that $\Gamma$ admits $F_2$
    as a random subgroup~\cite[Theorem~5.4,
      Proposition~5.8]{Monod:inv}.
  \item
    Let $M$ be an oriented closed connected hyperbolic $n$-manifold
    and let $\Gamma := \pi_1(M)$.  Then $\Gamma$ is finitely presentable
    and $\loneb n (\Gamma) \geq 1$ 
    as well as $\HH^n_b(\Gamma;\R) \not \cong 0$
    (Example~\ref{exa:hyp}, Proposition~\ref{prop:loneb}).
    Thus,
    \[ \hdb \Gamma \geq n
    \quad\text{and}\quad
    \cdb \Gamma \geq n.
    \]
    Such examples exist in all dimensions at least $2$.
  \item If $\Gamma$ is a group with~$\bcd \Gamma = \infty$, then
    $\cdb \Gamma =\infty$.  In particular, this applies to the groups
    from Section~\ref{sec:large}.
  \end{itemize}
\end{example}

However, as of now no example seems to be known of a group of finite
non-trivial bounded (co)homological dimension.

\begin{prop}[cohomological dimension as projective dimension]
  Let $\Gamma$ be a group. Then $\cdb \Gamma$ coincides with 
  the \emph{relatively projective dimension} of~$\Gamma$
  \begin{align*}
    \relprojdim \Gamma :=
    \inf \bigl\{ n \in \N \bigm|
    \;&\text{$\R$ admits a strong relatively projective}
    \\
    & \text{$\Gamma$-resolution of length~$\leq n$}\bigr\}
    \in \N \cup \{\infty\}.
  \end{align*}
\end{prop}
\begin{proof}
  We argue as in the case of classical cohomological dimension:
  
  We have $\cdb \Gamma \leq \relprojdim \Gamma$, because
  for all Banach $\Gamma$-modules~$V$ and all strong relatively
  projective $\Gamma$-resolutions~$C_* \longrightarrow \R$
  of~$\R$, we have~\cite[Theorem~3.7]{loehl1}
  \[ \HH^*_b(\Gamma;V) \cong \HH^*\bigl(B(C_*,V)^\Gamma\bigr).
  \]

  Conversely, we have $\relprojdim \Gamma \leq \cdb \Gamma$: If $C_*
  \longrightarrow\R$ is a strong relatively projective
  $\Gamma$-resolution of~$\R$ and if $n \in \N$ satisfies for all
  Banach $\Gamma$-modules~$V$ that
  \[ \HH_b^{n+1}(\Gamma;V) \cong 0,
  \]
  then the same argument as in the classical case shows that
  $\ker \partial_n$ is a relatively projective $\Gamma$-module.
  Therefore, $C_*$ can be truncated at degree~$n$. This
  shows that $\relprojdim \Gamma \leq n$.
\end{proof}

\begin{prop}[monotonicity of bounded (co)homological dimension]\label{prop:dimmono}
  Let $\Gamma$ be a group and let $\Lambda \leq \Gamma$ be a
  subgroup.  Then
  \[ \hdb \Lambda \leq \hdb \Gamma
  \quad\text{and}\quad
  \cdb \Lambda \leq \cdb \Gamma.
  \]
\end{prop}
\begin{proof}
  This is the usual Shapiro argument: Let $V$ be a
  Banach $\Lambda$-module and let $n \in \N$. 
  In bounded cohomology, we have~\cite[Proposition~10.1.3]{monod}
  \[ \HH^n_b(\Lambda;V)
     \cong
     \HH^n_b \bigl(\Gamma; B(\ell^1\Gamma,V)^\Lambda\bigr).
  \]

  In $\ell^1$-homology: There is a natural isomorphism
  \[ \lonech *(\Gamma) \mathbin{\overline\otimes_\Gamma} (\ell^1 \Gamma \mathbin{\overline\otimes_\Lambda} V)
     \cong \res^\Gamma_\Lambda \lonech * (\Gamma) \mathbin{\overline \otimes_\Lambda} V
  \]
  of Banach chain complexes; moreover, $\res^\Gamma_\Lambda
  \lonech*(\Gamma)$ is a strong relatively projective
  $\Lambda$-resolution of~$\R$.  Using the description of
  $\ell^1$-homology via projective
  resolutions~\cite[Theorem~3.7]{loehl1}, we thus obtain a natural
  isomorphism
  \[ \lonehom n (\Lambda;V)
     \cong
     \lonehom n (\Gamma; \ell^1\Gamma \mathbin{\overline\otimes_\Lambda} V).
     \qedhere
  \]
\end{proof}

%%%%%%%%%%%%%%%%%%%%%%%%%%%%%%%%%%%%%%%%%%%%%%%%%%%%%%%%%%%
% bib
 
\bibliographystyle{abbrv}
\bibliography{bacbib}

\begin{thebibliography}{10}

\bibitem{Agol}
I.~Agol.
\newblock The virtual {H}aken conjecture.
\newblock {\em Documenta Math.}, 18:1045--1087, 2013.

\bibitem{BDH}
G.~Baumslag, E.~Dyer, and A.~Heller.
\newblock The topology of discrete groups.
\newblock {\em J.~Pure Appl.\ Algebra}, 16(1):1--47, 1980.

\bibitem{BDM}
G.~Baumslag, E.~Dyer, and C.~Miller.
\newblock On the integral homology of finitely presented groups.
\newblock {\em Bull.\ Amer.\ Math.\ Soc.}, 4(3):321--324, 1981.

\bibitem{odd_fp}
G.~Baumslag and C.~F. Miller~III.
\newblock Some odd finitely presented groups.
\newblock {\em Bull.\ London Math.\ Soc}, 20(3):239--244, 1988.

\bibitem{BS}
G.~Baumslag and D.~Solitar.
\newblock Some two-generator one-relator non-{H}opfian groups.
\newblock {\em Bull.\ Amer.\ Math.\ Soc.}, 68(3):199--201, 1962.

\bibitem{BerrickDic}
A.~J. Berrick.
\newblock The acyclic group dichotomy.
\newblock {\em J.~Algebra}, 326(1):47--58, 2011.

\bibitem{Bouarich2}
A.~Bouarich.
\newblock Suites exactes en cohomologie born{\'e}e r{\'e}elle des groupes
  discrets.
\newblock {\em C.~R.~Acad.\ Sci.\ Paris S\'{e}r.~I Math.}, 320(11):1355--1359,
  1995.

\bibitem{Bouarich}
A.~Bouarich.
\newblock Exactitude {\`a} gauche du foncteur ${H}^n_b(-; \mathbb{R})$ de
  cohomologie born{\'e}e r{\'e}elle.
\newblock {\em Ann.\ Fac.\ Sci.\ Toulouse. $6^e$~s{\'e}rie}, 10(2):255--270,
  2001.

\bibitem{Brooks}
R.~Brooks.
\newblock Some remarks on bounded cohomology.
\newblock In {\em Riemann Surfaces Related Topics (AM-97), Volume 97}, pages
  53--64. Princeton University Press, 2016.

\bibitem{isometric}
M.~Bucher, M.~Burger, R.~Frigerio, A.~Iozzi, C.~Pagliantini, and M.~B.
  Pozzetti.
\newblock Isometric embeddings in bounded cohomology.
\newblock {\em J.~Topol.\ Anal.}, 6(1):1--25, 2014.

\bibitem{BIMW}
M.~Burger, A.~Iozzi, N.~Monod, and A.~Wienhard.
\newblock Bounds for cohomology classes.
\newblock {\em Enseign.\ Math.~(2)}, 54(1--2):3--189, 2008.
\newblock Guido's book of conjectures, A gift to Guido Mislin on the occasion
  of his retirement from ETHZ, June 2006, collected by I.~Chatterji.

\bibitem{lattices}
M.~Burger and N.~Monod.
\newblock Bounded cohomology of lattices in higher rank {L}ie groups.
\newblock {\em J.~Eur.\ Math.\ Soc.}, 1(2):199--235, 1999.

\bibitem{rigidity}
M.~Burger and N.~Monod.
\newblock Continuous bounded cohomology and applications to rigidity theory.
\newblock {\em Geom.\ Funct.\ Anal.\ GAFA}, 12(2):219--280, 2002.

\bibitem{calegari}
D.~Calegari.
\newblock {\em scl}, volume~20 of {\em MSJ Memoirs}.
\newblock Mathematical Society of Japan, Tokyo, 2009.

\bibitem{Camm}
R.~Camm.
\newblock Simple free products.
\newblock {\em J.~London Math.\ Soc.}, 28:66--76, 1953.

\bibitem{thompson}
J.~W. Cannon, W.~J. Floyd, and W.~R. Parry.
\newblock Introductory notes on {R}ichard {T}hompson's groups.
\newblock {\em Enseign.\ Math.}, 42:215--256, 1996.

\bibitem{knots}
R.~H. Crowell and R.~H. Fox.
\newblock {\em Introduction to knot theory}, volume~57.
\newblock Springer Science \& Business Media, 2012.

\bibitem{DGO}
F.~Dahmani, V.~Guirardel, and D.~Osin.
\newblock Hyperbolically embedded subgroups and rotating families in groups
  acting on hyperbolic spaces.
\newblock {\em Mem.\ Amer.\ Math.\ Soc.}, 245(1156):v+152, 2017.

\bibitem{Epstein_Fujiwara}
D.~B.~A. Epstein and K.~Fujiwara.
\newblock The second bounded cohomology of word-hyperbolic groups.
\newblock {\em Topology}, 36(6):1275--1289, 1997.

\bibitem{Faiziev}
V.~A. Faiziev and P.~K. Sahoo.
\newblock Remark on the second bounded cohomology of amalgamated products of
  groups.
\newblock {\em Sarajevo J.\ Math.}, 1(13):27--48, 2005.

\bibitem{binate}
F.~Fournier-Facio, C.~L{\"o}h, and M.~Moraschini.
\newblock Bounded cohomology and binate groups.
\newblock {\em arXiv.2111.04305, To appear in J. Austr. Math. Soc.}, 2021.

\bibitem{Frigerio}
R.~Frigerio.
\newblock {\em Bounded cohomology of discrete groups}, volume 227 of {\em
  Mathematical Surveys and Monographs}.
\newblock American Mathematical Soc., 2017.

\bibitem{FM:Grom}
R.~Frigerio and M.~Moraschini.
\newblock Gromov's theory of multicomplexes with applications to bounded
  cohomology and simplicial volume.
\newblock 2018.

\bibitem{FPS}
R.~Frigerio, M.~B. Pozzetti, and A.~Sisto.
\newblock Extending higher-dimensional quasi-cocycles.
\newblock {\em J.~Topol.}, 8(4):1123--1155, 2015.

\bibitem{Fujiwara:free}
K.~Fujiwara.
\newblock The second bounded cohomology of an amalgamated free product of
  groups.
\newblock {\em Trans.\ Amer.\ Math.\ Soc.}, 352(3):1113--1129, 2000.

\bibitem{gaboriau}
D.~Gaboriau.
\newblock Co{\^u}t des relations d'{\'e}quivalence et des groupes.
\newblock {\em Invent.\ Math.}, 139(1):41--98, 2000.

\bibitem{Ghys}
E.~Ghys.
\newblock Groupes d'hom\'{e}omorphismes du cercle et cohomologie born\'{e}e.
\newblock In {\em The {L}efschetz centennial conference, {P}art {III} ({M}exico
  {C}ity, 1984)}, volume~58 of {\em Contemp. Math.}, pages 81--106. Amer. Math.
  Soc., Providence, RI, 1987.

\bibitem{cohoT}
{\'E}.~Ghys and V.~Sergiescu.
\newblock Sur un groupe remarquable de diff{\'e}omorphismes du cercle.
\newblock {\em Comment.\ Math.\ Helv.}, 62(1):185--239, 1987.

\bibitem{gordon}
C.~M. Gordon.
\newblock Some embedding theorems and undecidability questions for groups.
\newblock In {\em Combinatorial and geometric group theory ({E}dinburgh,
  1993)}, volume 204 of {\em London Math.\ Soc.\ Lecture Note Ser.}, pages
  105--110. Cambridge University Press, 1995.

\bibitem{grabowski}
{\L}.~Grabowski.
\newblock Vanishing of {$l^2$}-cohomology as a computational problem.
\newblock {\em Bull.\ Lond.\ Math.\ Soc.}, 47(2):233--247, 2015.

\bibitem{Grigorchuk}
R.~I. Grigorchuk.
\newblock Some results on bounded cohomology.
\newblock In A.~J. Duncan, N.~D. Gilbert, and J.~Howie, editors, {\em
  Combinatorial and geometric group theory (Edinburgh, 1993)}, volume 204,
  pages 111--163. Cambridge University Press, 1995.

\bibitem{Grigorchuk:free}
R.~I. Grigorchuk.
\newblock Bounded cohomology of group constructions.
\newblock {\em Mat.\ Zamet.}, 59(4):546--550, 1996.

\bibitem{vbc}
M.~Gromov.
\newblock Volume and bounded cohomology.
\newblock {\em Inst.\ Hautes {\'E}tudes Sci.\ Publ.\ Math.}, 56:5--99 (1983),
  1982.

\bibitem{OH2}
T.~Hartnick and A.~Ott.
\newblock Bounded cohomology via partial differential equations, {I}.
\newblock {\em Geom.\ Topol}, 19:3603--3643, 2015.

\bibitem{Heuer_thesis}
N.~Heuer.
\newblock {\em Constructions in stable commutator length and bounded
  cohomology}.
\newblock PhD thesis, 2019.
\newblock available at \url{http://www.nicolausheuer.com/Thesis.pdf}.

\bibitem{Higman}
G.~Higman.
\newblock Subgroups of finitely presented groups.
\newblock {\em Proc.\ R.~Soc.\ Lond.~A}, 262(1311):455--475, 1961.

\bibitem{HNN}
G.~Higman, B.~H. Neumann, and H.~Neuman.
\newblock Embedding theorems for groups.
\newblock {\em J.~London Math.\ Soc.}, 1(4):247--254, 1949.

\bibitem{Hirshon}
R.~Hirshon.
\newblock Finitely generated groups ${L}$ with ${L}\cong {L}\times {M}$, ${M}
  \neq 1$, ${M}$ finitely presented.
\newblock {\em J.~Algebra}, 99(1):232--238, 1986.

\bibitem{prod_fp2}
R.~Hirshon.
\newblock Some properties of groups which allow homomorphisms onto their direct
  square.
\newblock {\em J.~Algebra}, 167(2):284--290, 1994.

\bibitem{prod_fp}
R.~Hirshon and D.~Meier.
\newblock Groups with a quotient that contains the original group as a direct
  factor.
\newblock {\em Bull.\ Austr.\ Math.\ Soc.}, 45(3):513--520, 1992.

\bibitem{Huber}
T.~Huber.
\newblock {\em Rotation quasimorphisms for surfaces}.
\newblock PhD thesis, ETH~Z{\"u}rich, 2012.
\newblock
  \url{http://e-collection.library.ethz.ch/eserv/eth:7059/eth-7059-02.pdf}.

\bibitem{Hull_Osin}
M.~Hull and D.~Osin.
\newblock Induced quasicocycles on groups with hyperbolically embedded
  subgroups.
\newblock {\em Algebr.\ Geom.\ Topol.}, 13(5):2635--2665, 2013.

\bibitem{ivanov}
N.~V. Ivanov.
\newblock Foundations of the theory of bounded cohomology.
\newblock volume 143, pages 69--109, 177--178. 1985.
\newblock Studies in topology,~V.

\bibitem{Ivanov_bac_covers}
N.~V. Ivanov.
\newblock Leray theorems in bounded cohomology theory.
\newblock arXiv:2012.08038, 2020.

\bibitem{Johnson2}
B.~E. Johnson.
\newblock Approximate diagonals and cohomology of certain annihilator {B}anach
  algebras.
\newblock {\em Amer.\ J.\ Math.}, 94(3):685--698, 1972.

\bibitem{Johnson}
B.~E. Johnson.
\newblock Cohomology in {B}anach algebras.
\newblock {\em Mem.\ Amer.\ Math.\ Soc.}, 127, 1972.

\bibitem{Jones}
J.~T. Jones.
\newblock Direct products and the {H}opf property.
\newblock {\em J.~Austr.\ Math.\ Soc.}, 17(2):174--196, 1974.

\bibitem{Lafont_Schmidt}
J.-F. Lafont and B.~Schmidt.
\newblock Simplical volume of closed locally symmetric spaces of non-compact
  type.
\newblock {\em Acta Math.}, 197(1):129--143, 2006.

\bibitem{Lafont_Wang}
J.-F. Lafont and S.~Wang.
\newblock Barycentric straightening and bounded cohomology.
\newblock {\em J.~Eur.\ Math.\ Soc.}, 21(2):381--403, 2018.

\bibitem{loehthesis}
C.~{L\"oh}.
\newblock {\em {\(\ell^1\)-{H}omology and {S}implicial {V}olume}}.
\newblock PhD thesis, 2007.
\newblock available at
  \url{https://miami.uni-muenster.de/Record/7ef3617a-65a5-41a9-92f1-b85e879e976c}.

\bibitem{loehl1}
C.~L{\"o}h.
\newblock Isomorphisms in {$\ell^1$}-homology.
\newblock {\em M{\"u}nster J.\ Math.}, 1:237--265, 2008.

\bibitem{Loeh}
C.~L{\"o}h.
\newblock A note on bounded-cohomological dimension of discrete groups.
\newblock {\em J.~Math.\ Soc.\ Japan}, 69(2):715--734, 2017.

\bibitem{lueckl2}
W.~L{\"u}ck.
\newblock {\em {$L^2$}-{I}nvariants: {T}heory and {A}pplications to {G}eometry
  and {$K$}-{T}heory}, volume~44 of {\em Ergebnisse der Mathematik und ihrer
  Grenzgebiete. 3.~Folge. A Series of Modern Surveys in Mathematics [Results in
  Mathematics and Related Areas. 3rd Series. A Series of Modern Surveys in
  Mathematics]}.
\newblock Springer, 2002.

\bibitem{MM}
S.~Matsumoto and S.~Morita.
\newblock Bounded cohomology of certain groups of homeomorphisms.
\newblock {\em Proc.\ Amer.\ Math.\ Soc.}, 94(3):539--544, 1985.

\bibitem{Meier}
D.~Meier.
\newblock Non-{H}opfian groups.
\newblock {\em J.~London Math.\ Soc.}, 2(2):265--270, 1982.

\bibitem{millerbook}
C.~F. Miller, III.
\newblock {\em On group-theoretic decision problems and their classification}.
\newblock Annals of Mathematics Studies, No.~68. Princeton University Press,
  1971.

\bibitem{Minasyan:Osin}
A.~Minasyan and D.~Osin.
\newblock Acylindrical hyperbolicity of groups acting on trees.
\newblock {\em Math.\ Ann.}, 362(3):1055--1105, 2015.

\bibitem{monod}
N.~Monod.
\newblock {\em Continuous bounded cohomology of locally compact groups}, volume
  1758 of {\em Lecture Notes in Mathematics}.
\newblock Springer-Verlag, Berlin, 2001.

\bibitem{Monod:inv}
N.~Monod.
\newblock An invitation to bounded cohomology.
\newblock In {\em International Congress of Mathematicians. Vol. II}, pages
  1183--1211. Eur. Math. Soc., Z{\"u}rich, 2006.

\bibitem{monod_sarithmetic}
N.~Monod.
\newblock On the bounded cohomology of semi-simple groups, {$S$}-arithmetic
  groups and products.
\newblock {\em J.~Reine Angew.\ Math.}, 640:167--202, 2010.

\bibitem{monod_F}
N.~Monod.
\newblock Lamplighters and the bounded cohomology of {T}hompson's group.
\newblock {\em Geom.\ Funct.\ Anal.}, 32(3):662--675, 2022.

\bibitem{monodnariman}
N.~Monod and S.~Nariman.
\newblock On the bounded cohomology of certain homeomorphism groups.
\newblock 2021.
\newblock arXiv:2111.04365.

\bibitem{coamenable}
N.~Monod and S.~Popa.
\newblock On co-amenability for groups and von {N}eumann algebras.
\newblock {\em C.~R.~Acad.\ Sci.\ Canada}, 25(3):82--87, 2003.

\bibitem{monodshalom}
N.~Monod and Y.~Shalom.
\newblock Orbit equivalence rigidity and bounded cohomology.
\newblock {\em Ann.\ of Math.~(2)}, 164(3):825--878, 2006.

\bibitem{BAc}
M.~Moraschini and G.~Raptis.
\newblock Amenability and acyclicity in bounded cohomology theory, 2021.
\newblock arXiv:2105.02821.

\bibitem{Nitsche}
M.~Nitsche.
\newblock Higher-degree bounded cohomology of transformation groups.
\newblock {\em arXiv:2105.08698}, 2021.

\bibitem{Rhemtulla}
A.~Rhemtulla.
\newblock Finitely generated non-{H}opfian groups.
\newblock {\em Proc.\ Amer.\ Math.\ Soc.}, 81(3):382--384, 1981.

\bibitem{rotman}
J.~J. Rotman.
\newblock {\em An introduction to the theory of groups}, volume 148 of {\em
  Graduate Texts in Mathematics}.
\newblock Springer, fourth edition, 1995.

\bibitem{Soma}
T.~Soma.
\newblock Bounded cohomology and topologically tame {K}leinian groups.
\newblock {\em Duke Math.~J.}, 88(2):357--370, 1997.

\bibitem{Thurston}
W.~P. Thurston.
\newblock {\em The geometry and topology of three-manifolds}.
\newblock Princeton University Princeton, NJ, 1979.

\bibitem{weinberger}
S.~Weinberger.
\newblock {\em Computers, rigidity, and moduli}.
\newblock M.~B.~Porter Lectures. Princeton University Press, 2005.
\newblock The large-scale fractal geometry of Riemannian moduli space.

\end{thebibliography}

\end{document}